\documentclass[11pt,reqno,tbtags,a4paper]{amsart}
\usepackage{amssymb}
\usepackage{url}
\usepackage[square,numbers]{natbib}
\bibpunct[, ]{[}{]}{;}{n}{,}{,}

\title%[]
{Non-fringe subtrees in conditioned Galton--Watson trees}

\date{6 March 2018}
\author{Xing Shi Cai}
\author{Svante Janson}
\thanks{Partly supported by the Knut and Alice Wallenberg Foundation}
\address{Department of Mathematics, Uppsala University, PO Box 480,
SE-751~06 Uppsala, Sweden}
\email{\href{mailto:xingshi.cai@math.uu.se}{xingshi.cai@math.uu.se},
    \href{mailto:svante.janson@math.uu.se}{svante.janson@math.uu.se}}
\urladdr{\url{http://www.math.uu.se/svante-janson}}

\subjclass[2010]{} 
\numberwithin{equation}{section}

\renewcommand\le{\leqslant}
\renewcommand\ge{\geqslant}

\allowdisplaybreaks

\theoremstyle{plain}% default
\newtheorem{theorem}{Theorem}[section]
\newtheorem{lemma}[theorem]{Lemma}

\theoremstyle{definition}
\newtheorem{example}[theorem]{Example}

\newtheorem{remark}[theorem]{Remark}

\theoremstyle{remark}

\newenvironment{romenumerate}[1][-10pt]{% optional argument changes indentation
\addtolength{\leftmargini}{#1}\begin{enumerate}% gives (i), (ii) etc.
 \renewcommand{\labelenumi}{\textup{(\roman{enumi})}}%
 }{\end{enumerate}}

\newcounter{oldenumi}
{\setcounter{oldenumi}{\value{enumi}}
\begin{romenumerate} \setcounter{enumi}{\value{oldenumi}}}
{\end{romenumerate}}

\newcounter{thmenumerate}

\newcounter{xenumerate}   %no left indentation; thus wider lines

\newcommand{\refT}[1]{Theorem~\ref{#1}}

\newcommand{\refL}[1]{Lemma~\ref{#1}}
\newcommand{\refR}[1]{Remark~\ref{#1}}
\newcommand{\refS}[1]{Section~\ref{#1}}
\newcommand{\refSs}[1]{Sections~\ref{#1}}
\newcommand{\refSS}[1]{Section~\ref{#1}}
\newcommand{\refSSs}[1]{Sections~\ref{#1}}
\begingroup
  \divide\count255 by 60
  \multiply\count255 by -60
  \advance\count255 by \time
  \ifnum \count255 < 10 \xdef\klockan{\the\count1.0\the\count255}
  \else\xdef\klockan{\the\count1.\the\count255}\fi
\endgroup

\newcommand{\sumno}{\sum_{n=0}^\infty}

\newcommand{\sumk}{\sum_{k=1}^\infty}

\newcommand{\sumn}{\sum_{n=1}^\infty}

\newcommand\set[1]{\ensuremath{\{#1\}}}

\newcommand\xpar[1]{(#1)}
\newcommand\bigpar[1]{\bigl(#1\bigr)}
\newcommand\Bigpar[1]{\Bigl(#1\Bigr)}

\newcommand\lrpar[1]{\left(#1\right)}

\def\rompar(#1){\textup(#1\textup)}    % usage: \rompar(...)

\def\xexp(#1){e^{#1}}

\newcommand\ntoo{\ensuremath{{n\to\infty}}}

\newcommand\mtoo{\ensuremath{{m\to\infty}}}

\newcommand\upto{\nearrow}
\newcommand\punkt{.\spacefactor=1000}    % om problem!
\newcommand\iid{i.i.d\punkt}    

\newcommand\eg{e.g\punkt}

\newcommand\cf{cf\punkt}

\newcommand{\aex}{a.e\punkt}
\newcommand{\tend}{\longrightarrow}
\newcommand\dto{\overset{\mathrm{d}}{\tend}}
\newcommand\pto{\overset{\mathrm{p}}{\tend}}

\newcommand\eqd{\overset{\mathrm{d}}{=}}

\newcommand\bbR{\mathbb R}

\newcommand\bbQ{\mathbb Q}
\newcommand\bbZ{\mathbb Z}

\newcounter{CC}
\newcommand{\CC}{\stepcounter{CC}\CCx} %new constant C_i
\newcommand{\CCx}{C_{\arabic{CC}}}     %repeats the last C_i
\newcommand{\CCdef}[1]{\xdef#1{\CCx}}     %defines #1 as the last C_i
\newcommand{\CCname}[1]{\CC\CCdef{#1}}    %new C_i and defines #1 as it
\newcounter{cc}
\newcommand\E{\operatorname{\mathbb E{}}}
\newcommand\e[1]{\operatorname{\mathbb E}\left[#1\right]}
\renewcommand\P{\operatorname{\mathbb P{}}}

\newcommand\Var{\operatorname{Var}}
\newcommand\Cov{\operatorname{Cov}}

\newcommand\Bi{\operatorname{Bi}}

\newcommand\ga{\alpha}
\newcommand\gb{\beta}
\newcommand\gd{\delta}

\newcommand\gam{\gamma}

\newcommand\gl{\lambda}

\newcommand\gss{\sigma^2}

\newcommand\cD{\mathcal D}
\newcommand\cE{\mathcal E}
\newcommand\cF{\mathcal F}
\newcommand\cG{\mathcal G}

\newcommand\cR{{\mathcal R}}
\newcommand\cS{{\mathcal S}}
\newcommand\cT{{\mathcal T}}

\newcommand\cX{{\mathcal X}}

\newcommand\cZ{{\mathcal Z}}

\newcommand\qw{^{-1}}

\newcommand\qqw{^{-1/2}}

\newcommand{\pgf}{probability generating function}

\newcommand\lhs{left-hand side}
\newcommand\rhs{right-hand side}

\newcommand\GW{Galton--Watson}
\newcommand\GWt{\GW{} tree}
\newcommand\cGWt{conditioned \GW{} tree}

\newcommand\rot{o}
\newcommand\ax[1]{a^{(#1)}}
\newcommand\ai{\ax1}
\newcommand\aii{\ax2}
\newcommand\bx[1]{b^{(#1)}}
\newcommand\bi{\bx1}

\newcommand\cFF{\widehat{\cF}}
\newcommand\FF{\widehat{F}}
\newcommand\dv[1]{D_+{#1}}
\newcommand\gamx{\gam^*}
\newcommand\gamxx{\gam}
\newcommand\gamy{\gam^{**}}
\newcommand\gamz{\zeta}
\newcommand\cemp{\cE_m^+}
\newcommand\cemm{\cE_m^-}
\newcommand\cxmp{\cX_m^+}
\newcommand\cxmm{\cX_m^-}
\newcommand\nmp{n_m^+}
\newcommand\nmm{n_m^-}
\newcommand\xm{X^-}
\newcommand\xp{X^+}
\newcommand\tta{T_a}
\newcommand\ttb{T_b}
\newcommand\ctn{\cT_n}
\newcommand\gf{\phi}
\newcommand\tphi{\tilde\phi}
\newcommand\ellm{_{m,\ell}}
\newcommand\mux{\mu_1}
\newcommand\gssx{\sigma_1^2}
\usepackage{graphicx}

\newcommand*{\Cdot}{\bullet} %\raisebox{-1.15ex}{\scalebox{3}{$\cdot$}}}

\usepackage{mathtools} 
\usepackage{array}

\usepackage{xcolor}
\definecolor{brightmaroon}{rgb}{0.76, 0.13, 0.28}
\definecolor{maroon}{rgb}{0.5, 0, 0}
\definecolor{webgreen}{rgb}{0, 0.5, 0}
\usepackage{hyperref}
\hypersetup{colorlinks, breaklinks, urlcolor=maroon, linkcolor=maroon, citecolor=webgreen} % Set link colors

\newcommand\myseqii[1]{(#1)_{i \ge 0}}
\newcommand\myseqi[1]{(#1_{i})_{i \ge 0}}
\newcommand\p[1]{\mathbb{P}\left(#1\right)}
\newcommand\ctns{\cT_{n}^{[s]}}
\newcommand\sF{F}

\newcommand\tdF{{F}}

\usepackage{silence}
\begin{document}

\begin{abstract}
    We study \(S(\cT_{n})\), the number of subtrees in a conditioned Galton--Watson tree of size
    \(n\). With two very different methods, we show that \(\log(S(\cT_{n}))\) has a Central Limit
    Law and that the moments of \(S(\cT_{n})\) are of exponential scale. 
\end{abstract}

\maketitle

We define the model which we study in \refS{Sdef}. Our main results are given in Section
\ref{S:intro}; the proofs can be found in Sections \ref{SpfLTN}, \ref{sec:mmt} and \ref{sec:gen}
respectively.  An extension is given in Section \ref{sec:size}.

\section{Definitions}\label{Sdef}

\subsection{Subtrees}

We consider only rooted trees.
We denote the node set of a rooted tree $T$ by $V(T)$, and the number of
nodes by $|T|=|V(T)|$.
We denote the root of $T$ by $\rot=\rot(T)$.
We regard the edges of a rooted tree as directed away from the root.

A \emph{(general) subtree} of a rooted tree $T$ is 
a subgraph $T'$ that is a tree. 
$T'$ is necessarily an induced subgraph, so
we may identify it with its node set $V'=V(T')$;
hence we can also define a 
subtree as any set of nodes that
forms a tree; in other words, any non-empty connected subset $V'$ 
of the node set
$V(T)$. 

Note that a  subtree $T'$ of $T$ has a unique node $\rot'$ of
smallest depth in $T$, and that all edges in $T'$ are directed away from
$\rot'$. We define $\rot'$ to be the root of $T'$.
Thus every  subtree $T'$ is itself a rooted tree, with the direction of
any edge agreeing with the direction in $T$.

A \emph{fringe subtree} is a subtree $T'$ that contains all children of any
node in it, i.e., if $v\in V'=V(T')$ then $w\in V'$ for every child $w$ of
$v$.
Equivalently, a fringe subtree is the tree $T_v$ consisting of all
descendants (in $T$) of some node $v\in V(T)$ (which becomes the root of
$T_v$). Hence the number of fringe subtrees of $T$ equals the number of
nodes of $T$.

Fringe subtrees are studied in many papers; often they are simply called
\emph{subtrees}.
 To avoid confusion,
we call the general subtrees studied in the
present paper \emph{non-fringe subtrees}.
(This is a minor abuse of notation, since 
fringe subtrees are examples of  non-fringe subtrees; the name 
should be interpreted as 
``not necessarily fringe''.)

A \emph{root subtree} of a rooted tree $T$ is non-fringe subtree $T'$ that
contains the root $o(T)$ (which then becomes the root of $T'$ too).
Equivalently, a root subtree is a non-empty 
set $V'\subseteq V(T)$ such that if $v\in V'$, then the parent of $v$ also
belongs to $V'$.

Let $\cS(T)$ be the set of non-fringe subtrees of $T$, and $\cR(T)$ the subset of 
root subtrees.
Let $S(T):=|\cS(T)|$ be the number of non-fringe subtrees of $T$, 
and $R(T):=|\cR(T)|$ the number of root subtrees.

Note that a non-fringe subtree of $T$ is a root subtree of a unique fringe
subtree $T_v$. Hence,
\begin{equation}\label{SR}
  S(T)=\sum_{v\in T} R(T_v).
\end{equation}
Furthermore, for any $v\in T$, $R(T_v)\le R(T)$, since we obtain an
injective map $\cR(T_v)\to\cR(T)$ by adding to each tree $T'\in\cR(T_v)$ the
unique path from $o$ to $v$. Consequently, using \eqref{SR},
\begin{equation}
  \label{SR2}
R(T)\le S(T) \le |T|\cdot R(T),
\end{equation}

\subsection{Conditioned Galton--Watson trees}

A \GW{} tree \(\cT\) is a tree in which each node is given a random number of child nodes, where the
numbers of child nodes are drawn independently from the same distribution \(\xi\) which is often
called the \emph{offspring distribution}.  
(We use $\xi$ to denote both the offspring distribution and a random
variable with this distribution.)
\GW{} trees were implicitly
introduced by \citet{bienayme} and
\citet{watson1875} for modeling the evolution of populations.

A conditioned \GW{} tree \(\ctn\) is a \GW{} tree conditioned on having size \(n\).  It is
well-known that \(\ctn\) encompasses many random tree models. For example, if
\(\p{\xi=i}=2^{-i-1}\), i.e., \(\xi\) has geometric \(1/2\) distribution, then \(\ctn{}\) is a
uniform random tree of size \(n\).  
Similarly, if \(\p{\xi=0}=\p{\xi=2}=1/2\), then \(\ctn{}\) is a
uniform random full binary tree of size \(n\).

As a result, the properties of \(\ctn\) has been well-studied. See,
\eg{}, \cite{SJ264} and the references there.
For fringe and non-fringe subtrees of \(\ctn\), see
\cite{SJ285, chyzak08, Cai2016Phd, cai16}.

\subsection{Simply generated trees}  
\label{sec:simple}

Let \(\myseqi{w}\) be a given sequence of nonnegative numbers, with $w_0>0$.
For a tree \(T\), 
let \(\dv(v)\) be the out-degree (number of children) of a node \(v\in T\),
and
define the \emph{weight} of
\(T\) by
\begin{equation} \label{weight}
    w(T) = \prod_{v \in T} w_{\dv(v)}.
\end{equation}
Let \(\ctns\) be a tree chosen at random from all ordered
trees of size \(n\) with probability proportional to
their weights. In other words,
\begin{equation} \label{simply}
    \p{\ctns = T} = 
    \frac{
        w(T)
    }{
        \sum_{T:|T|=n} w(T)
    }
    .
\end{equation}
We call \(\ctns\) a \emph{simply generated tree} with 
weight sequences \(\myseqi{w}\), and the generating function
\begin{equation}\label{gen:phi}
    \Phi(z) \coloneqq \sum_{i \ge 0} w_{i} z^{i}.
\end{equation}
its \emph{generator}.

Note that the \cGWt{}
\(\ctn\) with the offspring distribution \(\xi\) is 
the same as the
\(\ctns{}\) with the
weight sequence \(\myseqii{\p{\xi=i}}\).  In this case, the generator \(\Phi(z)\) is just the
probability generating function of \(\xi\).  
Hence, simply generated trees generalize conditioned \GW{} trees. 
On the other hand, given 
a sequence $(w_i)$ with generator $\Phi(z)$, any sequence with a generator
$a\Phi(bz)$ with $a,b>0$ yields the same \(\ctns\), 
and in many cases $a$ and $b$ can be chosen such that 
the new generator is a probability generating function, and then $\ctns$ is
a \cGWt. 
Consequently, simply generated trees and conditioned \GWt{s} are
essentially the same, and  we use in the sequel the notation \(\ctn\) for
both.
In particular,
see, \eg{}, \cite[Section 4]{SJ264},
a simply generated tree with generator
$\Phi(z)$ is equivalent to a \cGWt{} with offspring distribution
$\xi$ satisfying $\E\xi=1$ and $\E e^{t\xi}<\infty$ for some $t>0$,
if and only if
$\Phi(z)$ has a positive radius of convergence $R\in(0,\infty]$ and
    \begin{equation}
        \label{philm1}
        \lim_{z \upto R} 
        \frac{
            z \Phi'(z)
        }{
            \Phi(z)
        }
>1 
        .
    \end{equation}
Although the two formulations are equivalent under our conditions,
the formulation with simply generated trees is sometimes more convenient,
since it gives more flexibility in choosing a convenient $\Phi$; see for
example \refSS{sec:gf}.

For more on the connection between the two models,
see \cite[pp.~196--198]{Flajolet2009} and
\cite[Sections 2 and 4]{SJ264}.

\subsection{Some further notation}

If $v$ and $w$ are nodes in a tree $T$, then $v\prec w$ denotes that $v$ is
ancestor of $w$. 

We denote \(T'\) being a non-fringe (general) subtree of \(T\) by \(T' \subseteq T\) and \(T'\)
being a root subtree of \(T\) by \(T' \subseteq_{r} T\).

For a formal power series \(f(z)\coloneqq\sum_n{f_{n}} z^{n}\), we let \([z^{n}] f(z)\coloneqq f_n\).

\section{Main results}\label{S:intro}

We give two types of results in this paper, proved by two different methods.
First,
both $R(\ctn)$ and $S(\ctn)$ have an asymptotic log-normal
distribution, as conjectured by Luc Devroye (personal communication).

\begin{theorem}\label{TLN}
    Let $\ctn$ be a random conditioned \GWt{} of order $n$, 
    defined by some offspring distribution
    $\xi$ with $\E \xi=1$ and $0<\Var\xi<\infty$.
    Then there exist constants $\mu,\gss>0$ such that, as \ntoo,
    \begin{align}
        \frac{\log R(\ctn) - \mu n}{\sqrt n} &\dto N(0,\gss), \label{TLNR}
        \\
        \frac{\log S(\ctn) - \mu n}{\sqrt n} &\dto N(0,\gss),\label{TLNS}
    \end{align}
    where \(N(0,\gss)\) denotes the normal distribution with mean \(0\) and variance \(\gss\).
    Furthermore,
    \begin{align}
        \E [\log R(\ctn)]&=\E [\log S(\ctn)]+O(\log n) = n\mu + o\bigpar{\sqrt n},
        \label{ElogR}
        \\
        \Var [\log R(\ctn)]&=\Var [\log S(\ctn)]+o(n) = n\gss + o\xpar{ n}.
        \label{VlogR}
    \end{align}
\end{theorem}
\noindent The proof is given in \refS{SpfLTN}, and is based on a general theorem in
\cite{SJ285}. 
It is in principle possible to calculate $\mu$ and $\gss$ in \refT{TLN}, at
least numerically, see \refR{Rmugss}.

Secondly, if we also assume that \(\xi\) has a finite exponential moment
(a mild assumption satisfied by all standard examples),
then we can use
generating functions and singularity analysis to obtain asymptotics for the
mean and higher moments
of $R(\ctn)$.

\begin{theorem}\label{Tmom}
Let $\ctn$ be as in \refT{TLN}, and assume further that
$\E e^{t\xi}<\infty$ for some $t>0$.
Assume further that if $R\le\infty$ is the radius of convergence of the
probability generating function $\Phi(z):=\E z^\xi$, 
then 
$\Phi'(R):=  \lim_{z \upto R}  \Phi'(z) =\infty$.
   Then there exist sequences of numbers $\gamma_{m}>0$ and 
$1< \tau_1 < \tau_2 < \dots$
    such that for any fixed $m\ge1$, 
    \begin{equation}
        \label{eq:mmt3}
        \E R(\ctn)^m = \bigpar{1+O(n\qw)}\gamma_{m} \tau_m^n.
    \end{equation}
\end{theorem}

We will later use the formulation of simply generated trees. In this
language, \refT{Tmom} has the following, equivalent, formulation.

\begin{theorem}\label{Tmom2} 
Let $\ctn$ be a simply generated tree with generator $\Phi(z)$.
Let \(R \le \infty\) be the radius of convergence of \(\Phi(z)\).
Assume that $R>0$ and that
\begin{gather}
%\begin{align}
    \label{philm}
    \lim_{z \upto R} 
    \frac{
        z \Phi'(z)
    }{
        \Phi(z)
    }
    > 
    1
    ,
\\
\Phi'(R):=  \lim_{z \upto R}  \Phi'(z) 
=\infty
\label{philmx}.
\end{gather}
% \end{align}
Then \eqref{eq:mmt3} holds.
\end{theorem}

The proof of Theorems \ref{Tmom}--\ref{Tmom2} is given in
Section \ref{sec:mmt}.
 We first (\refSSs{sec:gf}--\ref{sec:sing})
illustrate the argument
by studying the simple case of full binary trees, where we do explicit
calculations.
(Similar explicit calculations could presumably be performed, \eg{}, for
full \(d\)-ary trees, or for ordered trees.) Then we give the proof for the
general case in \refSS{sec:mmt:proof}.
Note that the condition \eqref{philm} is the same as \eqref{philm1};
however, we need also the extra condition \eqref{philmx}. The latter
condition is a weak assumption that is satisfied in most applications, and
in particular if $R=\infty$, or if $\Phi(R)=\infty$. Nevertheless, this
extra condition (or some other) is necessary; we give in \refSS{SScounter}
an example showing that Theorems \ref{Tmom}--\ref{Tmom2} are not valid
without \eqref{philmx}.

For moments of $S(\ctn)$, we have by \eqref{SR2} the same exponential growth
$\tau_m^n$, but possibly also a polynomial factor.
 In fact, there is no such polynomial factor, and 
$\E S(\ctn)^m$ and $\E R(\ctn)^m$ differ asyptotically only by a constant
factor, as shown by the following theorem, proved in
\refS{sec:gen}.

\begin{theorem}\label{TS}
  Let $\ctn$ be as in Theorem \ref{Tmom} or \ref{Tmom2}.
Then, for any $m\ge1$,
    \begin{equation}
        \label{eq:S}
        \E S(\ctn)^m = \bigpar{1+O(n\qw)}\gamma'_{m} \tau_m^n,
    \end{equation}
where $\tau_m$ is as in \eqref{eq:mmt3} and $\gam_m'>0$. \\
More generally, for $m,\ell\ge0$,
    \begin{equation}
        \label{eq:RS}
   \E [R(\ctn)^\ell S(\ctn)^m] 
= \bigpar{1+O(n\qw)}\gamma'\ellm \tau_{\ell+m}^n,
    \end{equation}
for some $\gamma'\ellm>0$.
\end{theorem}
The constants $\gamma'\ellm$ can be calculated explicitly, see 
\eqref{gamma'}.

\begin{remark}\label{Rmom}
  We can express \eqref{TLNR} and \eqref{TLNS} by saying that $R(\ctn)$ and
  $S(\ctn)$ have the asymptotic distribution $LN(n\mu,n\gss)$.
Note that if $W\sim LN(n\mu,n\gss)$ exactly, so $W=e^Z$ with $Z\sim
N(n\mu,n\gss)$, then the moments of $W$ are given by
\begin{equation}
  \E W^m = \E e^{mZ} = e^{mn\mu+m^2n\gss/2}=e^{(m\mu+m^2\gss/2)n}.
\end{equation}
We may compare this to \refT{Tmom} and ask whether
\begin{equation}\label{??}
\tau_m\overset{?}=e^{m\mu+m^2\gss/2}.  
\end{equation}
It seems natural to guess that equality holds in \eqref{??};
however, we show in \refR{Rneq} 
that it does not hold, at least not for all $m$, in the
case of full binary trees. We therefore conjecture that, in fact, 
equality never holds in \eqref{??}.
This may seem surprising;
however, note that the same happens in the simpler
case $Y=e^X$ with $X\sim\Bi(n,p)$, with $p$ fixed. Then
$Y$ is asymptotically $LN(np,np(1-p))$ in the sense above,
but $\E Y^m=\E e^{mX} =\bigpar{1+p(e^m-1)}^n$ while if 
$W\sim LN(np,np(1-p))$, then $\E W^m=e^{(mp+m^2p(1-p)/2)n}$, with a different
basis for the $n$:th power.
\end{remark}

\section{Proof of \refT{TLN}}\label{SpfLTN}

  \begin{proof}[Proof of \refT{TLN}]
First, by \eqref{SR2}, $|\log S(\ctn)-\log R(\ctn)|\le\log n$, and thus 
\eqref{TLNR} and \eqref{TLNS} are equivalent. 
Similarly, the first inequalities in \eqref{ElogR} and \eqref{VlogR} hold,
using also Minkowski's inequality for the latter.
We consider in the rest of the
proof only $R(\ctn)$. 

Suppose that the root $o$ of $T$ has $D$ children $v_1,\dots,v_D$, and write
$T_i:=T_{v_i}$.
Then, a root subtree of $T$ consists of the root $o$ and, for each child
$v_i$, either the empty set or a root subtree of $T_i$. Consequently,
\begin{equation}
\label{rut}
R(T) = \prod_{i=1}^D \bigpar{R(T_i)+1}  .  
\end{equation}

Define 
\begin{equation}\label{FT}
F(T):=\log \bigpar{R(T)+1}=\log R(T)+O(1).
\end{equation}
Then \eqref{rut} implies
\begin{equation}
  F(T)=\log R(T) + \log \bigpar{1+R(T)\qw}
= \sum_{i=1}^D F(T_i)
+ \log \bigpar{1+R(T)\qw}.
\end{equation}
In other words, $F(T)$ is an additive functional with toll function
$f(T):=\log \bigpar{1+R(T)\qw}$, see \eg{} \cite[\S1]{SJ285}.

For any tree $T$, and any node $v\in T$, the path from the root $o$ to $v$
is a root subtree. Hence,
\begin{equation}
  R(T) \ge |T|,
\end{equation}
and as a consequence,
\begin{equation}\label{est}
0\le  f(T):=\log \bigpar{1+R(T)\qw} \le R(T)\qw \le |T|\qw.
\end{equation}
In particular, 
%$f(T)$ is uniformly bounded (by 1, for example), and 
we have the
deterministic bound $|f(\ctn)|\le 1/n$.
This bound implies that the conditions of \cite[Theorem 1.5]{SJ285} are
satisfied, and that theorem, together with the estimate in \eqref{FT}, 
yields \eqref{TLNR}, \eqref{ElogR} and
\eqref{VlogR}, for some $\mu,\gss\ge0$.
Furthermore, if $\cT$ is the (unconditioned) \GWt{} with offspring
distribution $\xi$, then
\begin{equation}\label{mu}
  \mu = \E f(\cT)>0.
\end{equation}
It remains only to verify that $\gss>0$. 
This is expected in all applications of 
 \cite[Theorem 1.5]{SJ285}, except trivial ones where $F(\ctn)$ is
 deterministic for all large $n$, but we do not know any general result; 
cf.\ \cite[Remark 1.7]{SJ285}.
In the present case, it can be verified as follows.

Consider a tree $T$. 
Denote the depth and out-degree (number of children) of a node $v\in T$ by
$d(v)$ and $\dv(v)$.
Fix a node $v\in T$, write $d=d(v)$, and let
the path from $o$ to $v$ be $o=v_0,v_1,\dots,v_d=v$.
By \eqref{rut}, we have for $j=0,\dots,d-1$,
\begin{equation}\label{ra}
R(T_{v_{j}}) = \ga_j \bigpar{R(T_{v_{j+1}})+1},   
\end{equation}
where $\ga_j$ is the
product of $R(T_w)+1$ over all children $w\neq v_{j+1}$ of $v_{j}$.
Note that each $R(T_w)\ge1$, and thus
\begin{equation}\label{ser}
  \ga_j \ge 2^{\dv(v_j)-1} \ge \dv(v_j).
\end{equation}
Define
\begin{equation}\label{beta}
  \gb(v):=\prod_{j=0}^{d-1} \ga_j,
\end{equation}
and
\begin{equation}\label{ala}
  \gamxx(v):= \sum_{j=1}^d\frac{ \gb(v_j)}{\gb(v)}
= \sum_{j=1}^d \prod_{k=j}^{d-1}\ga(v_k)\qw.
\end{equation}
Then repeated applications of \eqref{ra} (i.e., induction on $d$)
yield the expansion
\begin{equation}
  R(T)=R(T_{v_0})=\sum_{j=1}^d \gb(v_j)+\gb(v)R(T_v)
=\gb(v)\bigpar{R(T_v)+\gamxx(v)}.
\end{equation}
Hence, with
\begin{equation}\label{alax}
  \gamx(v):= \gamxx(v)+\gb(v)\qw=\sum_{j=0}^d\frac{ \gb(v_j)}{\gb(v)},
\end{equation}
we have
\begin{equation}\label{lys}
  F(T)=\log\bigpar{R(T)+1}
=\log\gb(v)+\log\bigpar{R(T_v)+\gamx(v)}.
\end{equation}
Define also 
\begin{equation}\label{ursa}
  \gamy(v):= \sum_{j=0}^d \prod_{k=j}^{d-1}\dv(v_k)\qw,
\end{equation}
and note that $\gamy(v)\ge\gamx(v)$ by \eqref{ser}--\eqref{alax}.

Now, let $T'$ be a modification of $T$, where the subtree $T_v$ is replaced
by some  tree $T_v'$, but all other parts of $T$ are left intact. Then all
$\ga_j$, $\gb(v_j)$, $\gamxx(v)$, $\gamx(v)$ and $\gamy(v)$ are the same for 
$T'$ as for $T$. 
Hence, 
if we  further assume that $R(T_v')<R(T_v)$, then \eqref{lys} yields
\begin{equation}\label{virgo}
  \begin{split}
  F(T)-F(T') &= \log\bigpar{R(T_v)+\gamx(v)} -\log\bigpar{R(T_v')+\gamx(v)}
\\&
\ge \log\bigpar{R(T_v)+\gamx(v)} -\log\bigpar{R(T_v)-1+\gamx(v)}
\\&
\ge \bigpar{R(T_v)+\gamx(v)}\qw
\ge \bigpar{R(T_v)+\gamy(v)}\qw.    
  \end{split}
\end{equation}

Next, fix an $\ell\ge2$ be such that $\P(\xi=\ell)>0$. 
%(I.e., nodes of out-degree $\ell$ appear.
Let $\tta$ be a tree where the root $o$ and two of its children have
out-degrees $\ell$, and all other nodes have out-degree 0 (i.e., they are leaves).
Similarly, let $\ttb$ be a tree where $o$, one of its children, and one of
its grandchildren have out-degree $\ell$, and all other nodes have out-degree 0.
Then both $\tta$ and $\ttb$ are  trees of order $3\ell+1$, 
and both are attained with positive probability by $\cT_{3\ell+1}$.
Furthermore, a simple
calculation using \eqref{rut} shows that
\begin{align}
  R(\tta)&=2^{\ell-2}(2^\ell+1)^2=2^{3\ell-2}+2^{2\ell-1}+2^{\ell-2}\label{tta}
  ,
\\
R(\ttb)&=2^{\ell-1}\bigpar{2^{\ell-1}(2^\ell+1)+1}=2^{3\ell-2}+2^{2\ell-2}+2^{\ell-1}
,
\label{ttb}
\end{align}
and thus $R(\tta)> R(\ttb)$. Consequently, the random variable $R(\cT_{3\ell+1})$
is not \aex{} equal to a constant.

Fix also a large constant $A$, to be chosen later, and say that a node 
$v\in T$ is \emph{good} if  $|T_v|=3\ell+1$ and $\gamy(v)\le A$.
Define the \emph{core} $T^*$ of $T$ as the subtree obtained by 
marking all good nodes in $T$ and then
deleting all descendants of them.
Note that 
adding back arbitrary trees of order $3\ell+1$ at each marked node of $T^*$
yield a tree $T'$ of the same order as $T$, and with the same good nodes,
because $|T_v|$ and $\gamy(v)$ are unchanged for every $v\in T^*$.
It follows that the random tree $\ctn$, conditioned on its core $\ctn^*=T^*$,
consists of $T^*$ with an added tree $T_v$ at each good (i.e., marked) node
of $T^*$, and that these added trees $T_v$ all have order $3\ell+1$ and are
independent copies of $\cT_{3\ell+1}$. 

Now suppose (in order to obtain a contradiction)
that $\gss=0$; then \eqref{TLNR} and \eqref{FT} show that 
$(F(\ctn)-\mu n)/\sqrt{n}\pto 0$. 
In particular,
\begin{equation}\label{ilu}
  \P\bigpar{|F(\ctn)-n\mu|>\sqrt n}\to0.
\end{equation}
We show in \refL{Lgood} below that there exists a constant $c>0$ such that,
for large $n$, 
$\ctn$ has with probability $\ge 1/2$ at least $cn$ good nodes.
Hence, \eqref{ilu} holds also if we condition on the existence of at least
$cn$ good nodes. Condition further on the core $\ctn^*$, and among the
possible cores $T^*$ of $\ctn$ with at least $cn$ good nodes, choose one that
minimizes
$ \P\bigpar{|F(\ctn)-n\mu|>\sqrt n\mid \ctn^*=T^*}$.
For each $n$, fix this choice $T^*=T^*(n)$, and note that
\begin{multline}\label{leu}
         \P\bigpar{|F(\ctn)-n\mu|>\sqrt n\mid \ctn^*=T^*}
\\
\le  \P\bigpar{|F(\ctn)-n\mu|>\sqrt n \mid \text{at least $cn$ good nodes}}\to0.
\end{multline}
Let $m$ be the number of good (i.e., marked) nodes in $T^*=T^*(n)$ and label
these $v_1,\dots,v_m$. 
Condition on $\ctn^*=T^*$.
Then, as noted above, 
$\ctn$ consists of $T^*$ with a tree $T_i$ added at $v_i$, for each $i$,
and these trees $T_1,\dots,T_m$ are $m$ independent copies of
$\cT_{3\ell+1}$.
Let $X_i:=R(T_i)$; thus $X_1,\dots,X_m$ are \iid{} random variables with
some fixed distribution. Furthermore, repeated applications of \eqref{rut}
show that $R(\ctn)$ is a function (depending on $T^*(n)$) of
$X_1,\dots,X_m$. Hence, 
by \eqref{FT}, we have, still conditioning on $\ctn^*=T^*$,
\begin{equation}\label{leo}
  F(\ctn) = g_n(X_1,\dots,X_m),
\end{equation}
for some function $g_n$.
Consequently, writing $Y_m:=g_n(X_1,\dots,X_m)$, we have by \eqref{leu}
\begin{equation}
  \P\bigpar{|Y_m-n\mu|>\sqrt n}
=          \P\bigpar{|F(\ctn)-n\mu|>\sqrt n\mid \ctn^*=T^*}\to0,
\end{equation}
as \ntoo. Recalling that $m\ge cn$, this implies
\begin{equation}\label{luna}
  \P\bigpar{|Y_m-n\mu|> c\qqw \sqrt m} \to 0.
\end{equation}

We now obtain the sought contradiction from \eqref{lna} in \refL{LNA} below.
(To be precise, we use a relabelling. 
We have $m=m(n)\to\infty$ as \ntoo; we may select a
subsequence with increasing $m$ and consider this sequence only, relabelling
$g_n$ as $g_m$.)
Note that in this application of \refL{LNA}, $S$ is a finite set of integers
(the range of $R(\cT_{3\ell+1})$).
The conditions of \refL{LNA} are satisfied:
by \eqref{tta}--\eqref{ttb}, we can find $s$ such that $0< \P(X_1\le s)<1$;
furthermore, 
\eqref{lna0} holds (under the stated condition) 
with $\gd:=(2^{3\ell+1}+A)\qw$
by \eqref{virgo}, since $\gamy(v_i)\le A$ by
the definition of good vertices and $R(T_v)\le 2^{|T_v|}=2^{3\ell-1}$.

This completes the proof that $\gss>0$, given the lemmas below.
  \end{proof}

  \begin{lemma}\label{Lgood}
With notations as above, there exists $A<\infty$ and $c>0$ such that,
for large $n$,
$\P\bigpar{\ctn \text{ has at least $cn$ good nodes}} \ge1/2$.
  \end{lemma}

  \begin{proof}
Note first that if 
$\P(\xi=1)=0$, and thus
$\ctn$ has no nodes of out-degree 1, 
then this is easy. In this case, \eqref{ursa} yields 
$\gamy(v)\le \sum_{j=0}^d 2^{j-d}<2$ for every $v$, 
since $\dv(v_k)\ge2$ for
each $v_k$. Taking $A=2$, every node $v$ with $|T_v|=3\ell+1$ is thus good.
If $n_{3\ell+1}(\ctn)$ denotes the number of these nodes in $\ctn$,
then  
\begin{equation}
n_{3\ell+1}(\ctn) /n \pto \P\bigpar{|\cT|=3\ell+1}>0,
\qquad\text{as \ntoo},
\end{equation}
and thus 
\begin{equation}\label{aquila}
\P\bigpar{n_{3\ell+1}(\ctn)>cn}\to 1  
\end{equation}
for any $c< \P(|\cT|=3\ell+1)$.

In general, \eqref{aquila} still holds, but there is no uniform bound on
$\gamy(v)$, as is shown by the case of a long path, and it remains to show
that $\gamy(v)$ is bounded for sufficiently many nodes.
We define, for a given tree $T$ and any pair of nodes $v,w$ with $v\preceq
w$,
\begin{equation}\label{pi}
  \pi(u,v):=\prod_{u\preceq w\prec v}\dv(w)\qw.
\end{equation}
We then can rewrite \eqref{ursa} as
\begin{equation}\label{gamy}
  \gamy(v)=\sum_{u\preceq v}\pi(u,v),
\end{equation}
and we define also the dual sum
\begin{equation}\label{gamz}
  \gamz(v)=\sum_{w\succeq v}\pi(v,w).
\end{equation}
Note that $\gamz(v)$ is a functional of the fringe subtree $T_v$.
We write $\gamz(T):=\gamz(o)$, where $o$ is the root of $T$;
then for an arbitrary node $v\in T$, $\gamz(v)=\gamz(T_v)$.

We may also note, although we do not use this explicitly, 
that $\gamz(v)$ has a natural
interpretation: $\pi(v,w)$ is the probability that a random walk, started
at $v$ and at each step choosing a child uniformly at random, will pass
through $w$. Hence, $\gamz(v)$ is the expected length of this random walk.

If the root $o$ of $T$ has $D$ children $v_1,\dots,v_D$, and the
corresponding 
fringe trees are denoted $T_1,\dots,T_D$, then
\begin{equation}\label{aries}
  \gamz(T)=\sum_{w\in T}\pi(o,w) = 1+\sum_{i=1}^{D} \sum_{w\in T_i} D\qw\pi(v_D,w)
=1+D\qw\sum_{i=1}^{D} \gamz(T_i).
\end{equation}
We apply this with $T=\ctn$, the \cGWt.
Note that conditioned on the root degree $D$, and the sizes $n_i:=|T_i|$ of
the subtrees, each $T_i$ is a conditioned \GWt{} $\cT_{n_i}$.
Consequently, \eqref{aries} yields
\begin{equation}\label{taurus}
\E\bigpar{\gamz(\ctn)\mid D, n_1,\dots,n_D}
=1+D\qw\sum_{i=1}^{D} \E \gamz(\cT_{n_i}).
\end{equation}
We claim that 
\begin{equation}\label{cca}
  \E \gamz(\ctn)\le \CCname{\CCa}
\end{equation}
for some constant $\CCa$ and all $n$.
We prove this by induction, assuming that \eqref{cca} holds for all smaller
$n$. Note also that if $|T|=1$, then $\gamz(T)=1$.
Hence, by \eqref{taurus} and the induction hypothesis, 
if $D_1:=|\set{i:n_i=1}|$, the number of children
of $o$ that are leaves, then
\begin{equation}\label{juno}
  \begin{split}
    \E\bigpar{\gamz(\ctn)\mid D, n_1,\dots,n_D}
&\le1+D\qw \bigpar{(D-D_1)\CCa+D_1}
\\
&=\CCa + 1 -D_1(\CCa-1)/D.
  \end{split}
\end{equation}
and hence
\begin{equation}\label{hera}
  \begin{split}
    \E\gamz(\ctn)
\le\CCa + 1 -(\CCa-1)\E(D_1/D),
  \end{split}
\end{equation}
where $D_1$ and $D$ are calculated for the random tree $\ctn$.
As \ntoo, the distribution of the pair $(D,D_1)$ converges to the $(\hat
D,\hat D_1)$, the same quantities for the random limiting infinite tree
$\hat T$, see for example \cite[Section 5 and Theorem 7.1]{SJ264}.
Hence, using bounded convergence, $\E(D_1/D)\to\E(\hat D_1/\hat D)>0$ as \ntoo.
Since $\P(D_1>0)>0$ for every $n$, and thus $\E(D_1/D)>0$ for every $n$,
it follows that there exists a constant $c_1>0$ such that for every $n$,
$\E(D_1/D)\ge c_1$. If we choose $\CCa=1+1/c_1$, then \eqref{hera} yields
$\E\gamz(\ctn)\le\CCa$, which verifies the induction step.
Hence, \eqref{cca} holds for all $n$.

Next, let, for any tree $T$,
\begin{equation}\label{ZT}
  Z(T):=\sum_{v\in T} \zeta(T_v),
\end{equation}
the additive functional with toll function $\zeta(T)$.
It follows from \eqref{cca} that
\begin{equation}
    \E{\zeta(\cT)}
    =
    \sum_{n \ge 1}
    \p{|\cT|=n} \e{\zeta(\cT_n)}
    \le \CCa,
    \label{aries1}
\end{equation}
where \(\cT\) denotes an unconditioned \GWt.
By \cite[Remark 5.3]{SJ285}, it follows from \eqref{aries1} and \eqref{cca}
that
% \(\E \zeta(\cT)=1/\P(\xi=0)\)
\begin{equation}\label{Zctn1}
  \E Z(\ctn)
  \sim
  n \E \zeta(\cT)
  =
 O( n)
  .
\end{equation}
Thus there exists a constant \(\CCname{\CCb}\) such that for all \(n \ge 1\),
\begin{equation}\label{Zctn}
  \E Z(\ctn)
  \le
  {\CCb} n
  .
\end{equation}
Consequently, by Markov's inequality, with probability $\ge \frac23$,
\begin{equation}\label{Zctn2}
  Z(\ctn) \le 3\CCb n.
\end{equation}

For any tree $T$, \eqref{ZT} and \eqref{gamy}--\eqref{gamz} yield
\begin{equation}
  Z(T)=\sum_{v\in T} \zeta(v)=\sum_{v,w:v\preceq w}\pi(v,w)
=\sum_{w\in T} \gamy(w).
\end{equation}
Hence, if we choose $A:=6\CCb/c$, then \eqref{Zctn2} implies that
at most $3\CCb n/A = cn/2$ nodes $w$ in $\ctn$ satisfy $\gamy(w)>A$, and
hence at least $n_{3\ell+1}(\ctn)-cn/2$ nodes are good.
This and \eqref{aquila} show that, with probability $\frac{2}3+o(1)$,
$\ctn$ has at least $cn/2$ good nodes.
  \end{proof}

  \begin{remark}
As the proof shows, the probability $1/2$ in \refL{Lgood} can be replaced by
any number $<1$. We conjecture that in fact, for suitable $A$ and $c$, the
probability tends to 1.    
  \end{remark}

  \begin{remark}
If we assume that the offspring distribution $\xi$ has an exponential
moment, so that its \pgf{} has radius of convergence $>1$, then one can
alternatively derive \eqref{cca} and \eqref{Zctn}, and precise asymptotics, 
using  generating functions. We leave this to the reader.     
  \end{remark}

  \begin{lemma}\label{LNA}
Let $X_1,X_2,\dots$ be \iid{} random variables, with values in some set
$S\subseteq\bbR$.
    Let $Y_m=g_m(X_1,\dots,X_m)$, for some functions $g_m:S^m\to\bbR$, 
$m\ge1$,
and    assume that there is a number $s$ and a $\gd>0$ 
such that
$0<\P(X_1\le s)<1$ and
that 
\begin{equation}\label{lna0}
  g_m\bigpar{y_1,\dots,y_{j-1},y_j',y_{j+1},\dots,y_m}
\ge   g_m\bigpar{y_1,\dots,y_{j-1},y_j,y_{j+1},\dots,y_m}+ \gd,
\end{equation}
for any $m$,
$j\le m$,
$y_1,\dots,y_m\in S$ and 
$y_j'\in S$,
such that $y_j\le s$ and $y_j'>s$.
 
Then, for any constant $B$ and any sequence $\mu_m$,
\begin{equation}\label{lna}
  \limsup_{\mtoo}\P\bigpar{|Y_m-\mu_m|\le B\sqrt m} <1.
\end{equation}
  \end{lemma}

  \begin{proof}
First, by replacing $g_m$ by $g_m-\mu_m$, we may assume that $\mu_m=0$.

If \eqref{lna} does not hold, then, by restricting attention to a
subsequence, we may assume $\P\bigpar{|Y_m|\le B\sqrt m} \to1$,
%i.e., $|Y_m|\le B\sqrt m$ \whp, 
as \mtoo.

    Let $N_m:=|\set{i:X_i> s}|$. Thus $N_m$ has a binomial distribution
    $\Bi(m,p)$, where $p:=\P(X_1>s)\in (0,1)$.
Fix a large number $K>0$, and define the events 
$\cemp:=\set{N_m>mp+K\sqrt m}$ and $\cemm:=\set{N_m<mp-K\sqrt m}$.
By the central limit theorem for the binomial distribution, $\P(\cemp)\to q$
and $\P(\cemm)\to q$ for some $q>0$, and thus our assumption implies that
\begin{equation}
\P\bigpar{|Y_m|\le B\sqrt m\mid\cemp} \to1,
\qquad
\P\bigpar{|Y_m|\le B\sqrt m\mid\cemm} \to1.
\end{equation}
Hence we can find integers $\nmp$ and $\nmm$ with 
$0\le\nmm<mp-K\sqrt m<mp+K\sqrt m<\nmp\le n$ such that
\begin{equation}\label{thr}
\P\bigpar{|Y_m|\le B\sqrt m\mid N_m=\nmp} \to1,
\qquad
\P\bigpar{|Y_m|\le B\sqrt m\mid N_m=\nmm} \to1.
\end{equation}
(Choose \eg{} $n_m^\pm$ as the integers in the allowed ranges that maximize
these probabilities.) 

Let $\cxmm=(\xm_1,\dots,\xm_m)$ be a random vector with the distribution of 
$\bigpar{(X_i)_1^m\mid N_m=\nmm}$.
By construction, a.s., 
exactly $\nmm$ of the variables $\xm_i$ satisfy
$\xm_i>s$, and thus $m-\nmm$ satisfy $\xm_i\le s$. Select $\nmp-\nmm$ of the
latter variables, chosen uniformly at random (independent of everything
except the set of indices $\set{i:\xm_i\le s}$),
and replace these by variables $\xp_i$ that are \iid{} copies of the random
variable $\xp:=\bigpar{X_1\mid X_1>s}$ (and independent of everything
else). Denote the result by $\cxmp$;
then
$\cxmp\eqd\bigpar{(X_i)_1^m\mid N_m=\nmp}$.
Consequently, by \eqref{thr},
\begin{equation}\label{try}
\P\bigpar{\left| g_m\bigpar{\cxmm} \right| \le B\sqrt m} \to1,
\qquad
\P\bigpar{\left| g_m\bigpar{\cxmp} \right| \le B\sqrt m} \to1.
\end{equation}
Hence, 
\begin{equation}\label{tryx}
\P\bigpar{|g_m\bigpar{\cxmm}-g_m\bigpar{\cxmp}|\le 2B\sqrt m} \to1.
\end{equation}
On the other hand, \eqref{lna0} and the construction imply that
\begin{equation}
  g_m\bigpar{\cxmp} - g_m\bigpar{\cxmm} \ge \bigpar{\nmp-\nmm}\gd
> 2K\sqrt{m}\gd.
\end{equation}
Choosing $K=B\gd\qw$, we obtain a contradiction with \eqref{tryx}.
%the set of indices $\cI_m:=\set{i:\xm_i>a}$ has \as{}
%exactly $\nmm$ elements, and thus 
%the complement $\cIc_m:=[m]\setminus\cI_m$ has $|\cIc_m|=m-\nmm$.
%Select $\nmp-\nmm$ of the 
%latter variables, chosen uniformly at random (independent of everything
  \end{proof}

  \begin{remark}\label{Rmugss}
The constant $\mu$ equals $\E f(\cT)$ by \eqref{mu};
we do not know any explicit closed form expression for $\mu$, 
but it seems possible to use \eqref{mu} for numerical calculation of $\mu$
for a given offspring distribution.
(Note that, by \eqref{est}, 
$f(T)\le R(T)\qw$, which typically
decreases exponentially in the size of $T$, so convergence ought to be
rather fast.)
For $\gss$, \cite[(1.17)]{SJ285} gives the formula
\begin{equation}
  \gss=2\E\bigpar{f(\cT)(F(\cT)-|\cT|\mu)}-\Var[f(\cT)]-\mu^2/\Var(\xi).
\end{equation}
Again, we do not know any closed form expression, but numerical calculation
should be possible.
  \end{remark}

\section{Moments of the number of root subtrees}\label{sec:mmt}

In this section we prove \refT{Tmom2}, using generating functions
and the language of simply generated trees; note that this also shows the 
equivalent \refT{Tmom}.
In  \refSSs{sec:gf} and \ref{sec:sing}, we study a simple example of
simply generated trees to 
illustrate the main idea behind Theorem \ref{Tmom2}; in this example we
derive explicit formulas for some generating functions.
%Our argument works for other simply generated trees given by other $\Phi$ 
%(but in general we do not find explicit formulas) 
The proof for the general case is postponed to  \refSS{sec:mmt:proof};
it uses the same argument
(but in general we do not find explicit formulas).

\subsection{An example: full binary trees}
\label{sec:gf}

Consider 
as an example
the simply generated tree \(\ctn\) with the generator $\Phi(z):=1+z^2$.
Then $\ctn$ is a uniformly random full binary tree of order $n$.  (Provided $n$ is
odd; otherwise, such trees do not exist.) Note that \(\Phi(z)\) satisfies the conditions of
\refT{Tmom2}.
(Note that we have chosen a generator that is not a probability generating
function; the corresponding offspring distribution $\xi$ has 
probability generating function $\frac12(1+z^2)$, and thus
$\P(\xi=0)=\P(\xi=2)=\frac12$; this generator would lead to similar
calculations and the same final result.)

A \emph{combinatorial class} is a finite or countably infinite set on which
a size function of range
\(\bbZ_{\ge 0}\) is defined. For a combinatorial class \(\cD\) and an element \(\delta \in \cD\), let
\(|\delta|\) denote its size.  The \emph{generating function} of \(\cD\) is
defined by 
\begin{align} %\show\coloneqq
    D(z) 
    \coloneqq \sum_{\delta \in \cD} 
    z^{|\delta|}
    =
    \sumno
    d_{n}
    z^{n}
    ,
    \label{eq:A:gf}
\end{align}
where \(d_{n}\) denotes the number of elements in \(\cD\) with size \(n\).
It encodes all the information of \( (d_{n})_{n \ge 0}\) and is a powerful tool to get asymptotic
approximations of \(d_{n}\).

Let \(\cZ = \{\Cdot\}\) denote the combinatorial class of node, which contains only one element
\(\Cdot\) since we are considering unlabelled trees. Let \(|\Cdot|=1\).  
Then the generating function of \(\cZ\) is simply \(z\).
Let \(\cF_{0}\) denote the combinatorial class of full binary trees. For \(T \in \cF_{0}\), we let \(|T|\) be the total
number of nodes in \(T\). Since \(T\) is a binary tree, it must be either a node, or a node together
with a left subtree \(T_{1}\) and a right subtree \(T_{2}\), with \(T_{1}, T_{2} \in \cF_{0}\). This
can be formalized by the symbolic language developed by \citet[p.~67]{Flajolet2009} as
\begin{align}
    \cF_{0} 
    = 
    \cZ + \cZ \times \cF_{0} \times \cF_{0}
    ,
    \label{eq:F0:sym}
\end{align}
with \(+\) denotes ``or'' and \(\times\) denotes ``combined with''.

Let \(F_{0}(z)\) denote the generating function of \(\cF_{0}\), i.e., 
\begin{equation}\label{F0}
F_0(z):=\sum_{T}z^{|T|}  
=\sumn a_n z^n
,
\end{equation}
where $a_n$ is the number of full binary trees of order $n$.
Then the definition \eqref{eq:F0:sym} directly
translates into the functional equation
\begin{align}
    F_{0}(z)
    = 
    z + z \times F_{0}(z) \times F_{0}(z)
    =
    z \Phi(F_{0}(z))
    ,
    \label{F0a}
\end{align}
with the explicit solution
\begin{equation}\label{F0x}
  F_0(z)=\frac{1-\sqrt{1-4z^2}}{2z}.
\end{equation}

%Let 
%\begin{equation}\label{F0}
%F_0(z):=\sum_{T}z^{|T|}  
%=:\sumn a_n z^n
%\end{equation}
%be the generating function for all full binary trees (with the
%weight $z^{|T|}$ for a tree $T$, counting all nodes);
%thus $a_n$ is the number of full binary trees of order $n$.
%Then, as is well-known
%\begin{equation}\label{F0a}
%  F_0(z)=z\Phi(F_0(z))=z+zF_0(z)^2
%\end{equation}
%with the explicit solution
%\begin{equation}\label{F0x}
%  F_0(z)=\frac{1-\sqrt{1-4z^2}}{2z}.
%\end{equation}

To compute \(\E R(T_{n})\), we consider a pair  \( (T,T')\) in which
\(T\) is a full binary tree and \(T'\) is a rooted subtree of \(T\) painted with color \(1\). Let \(\cF_{1}\) be the
combinatorial class of such partially colored full binary trees,
with
\(|(T,T')| = |T|\). Let
\(F_{1}(z)\) be the generating function of \(\cF_{1}\), i.e.,
\begin{equation}\label{F1}
F_1(z)
:=
\sum_{T}\sum_{T'\subseteq_{r} T} z^{|T|}  
= 
\sum_{T}R(T) z^{|T|}  
=:
\sumn \ai_n z^n
.
\end{equation}
Then, for any (odd) $n$,
\begin{equation}\label{ER}
  \E R(\ctn)= \ai_n/a_n.
\end{equation}

For a tree $T$ in
\(\cF_{1}\), its root $o$ is always colored. 
Every subtree $T_v$ where $v$ is a child of $o$ (so $d(v)=1$) can be
either itself a partially colored tree (an element of \(\cF_{1}\)) or
an uncolored tree (an element of \(\cF_{0}\)). Thus, we have the following symbolic specification
\begin{align}
    \cF_{1} 
    = 
    \cZ + \cZ \times (\cF_{0} + \cF_{1}) \times \left( \cF_{0} + \cF_{1} \right)
    =
    \cZ \Phi(\cF_{0} + \cF_{1})
    .
    \label{eq:F1:sym}
\end{align}
%To compute $F_1(z)$, we consider a pair $(T,T')$ with $T'\subseteq T$ as a
%tree $T$ with the nodes of $T'$ marked. The root $o$ is always marked.
%Any child of the root may be marked or not; if it is marked this branch
%contributes $F_1(z)$ and if the child is not marked, then the branch
%contributes $F_0(z)$. 
%\rem{This is badly written and must be expressed better,
%for example by the Flajolet calculaus as you did it!}
Consequently, using \eqref{F0a},
\begin{equation}\label{F1a}
  \begin{split}
  F_1(z)
&=z\Phi\bigpar{F_1(z)+F_0(z)}
=z+z\bigpar{F_0(z)+F_1(z)}^2
\\&
=F_0(z)+2zF_0(z)F_1(z)+zF_1(z)^2.    
  \end{split}
\end{equation}
with the explicit solution
\begin{equation}\label{F1x}
  \begin{split}
      F_1(z)&=\frac{1-\sqrt{1-4z(z+F_0(z))}}{2z}-F_0(z)
\\&
=\frac{1-\sqrt{2\sqrt{1-4z^2}-1-4z^2}}{2z}-F_0(z).
  \end{split}
\end{equation}

For the second and higher moments we argue similarly.
For \(m \ge 1\), we consider a \( (m+1)\)-tuple \( (T,T_{1}',\cdots,T_{m}')\) in which \(T\) is a
full binary tree and \(T_{1}',\cdots,T_{m}'\) are \(m\) root subtrees of \(T\) with \(T_{i}'\) painted with
color \({i}\). (Note that \(T_{1}',\cdots,T_{m}'\) are not necessarily
distinct.
Note also that a node may have several colors.)
Let \(\cF_{m}\) be the combinatorial class of
such partially \(m\)-colored trees. Let \( |(T,T_{1},\cdots,T_{m}')| = |T|\). 
Let \(F_{m}(z)\) be the generating function of \(\cF_{m}\), i.e.,
\begin{equation}\label{Fm}
F_m(z)
:=
\sum_{T}\sum_{T'_1,\dots,T_m'\subseteq_{r} T} z^{|T|}  
= 
\sum_{T}R(T)^m z^{|T|}  
=:
\sumn \ax{m}_n z^n
.
\end{equation}
Then, for any (odd) $n$,
\begin{equation}\label{ERm}
  \E R(\ctn)^m= \ax{m}_n/a_n.
\end{equation}

The root $o$ of a tree in $\cF_m$
is always painted by all $m$ colors.  
Every subtree $T_v$ where $v$ is a child of $o$
is itself a partially
$C$-colored tree for some set of colors $C\subseteq [m]:=\set{1,\dots,m}$.
Let, for a given (finite) set of colors $C$, $\cF_C$ be the class of 
partially $C$-coloured trees, defined analogously to $\cF_m$,
and note that there is an obvious isomorphism
$\cF_C\cong\cF_{|C|}$. Furthermore, let $\cFF_m:=\bigcup_{C\subseteq[m]}\cF_{C}$.
Taking into account that there are \(\binom{m}{k}\) ways to choose \(k\)
colors out of \(m\), we thus have  the equations
\begin{align}
    \cF_{m} 
&    = 
    \cZ + \cZ \times 
\cFF_m
    \times
\cFF_m
    =
    \cZ \Phi\bigpar{
\cFF_m
    },
   \label{eq:Fm:sym}
\\
\cFF_m 
&= 
       \sum_{k=0}^{m} \binom{m}{k} \cF_{k}
. 
    \label{eq:FFm}
\end{align}
%
%The root $o$ is always marked by all $m$ marks.
%Any child of the root may be marked by an arbitrary subset of the marks; if
%it is marked by $k$ marks, then this branch
%contributes $F_k(z)$.
Consequently, for the corresponding generating functions,
\begin{equation}\label{Fma}
  \begin{split}
  F_m(z)
&=z\Phi\bigpar{\FF_m(z)}
=z+z\lrpar{\sum_{k=0}^m \binom mk F_k(z)}^2,
%\\&
  \end{split}
\end{equation}
which determines every $F_m(z)$ by recursion, solving a quadratic equation
in each step.
Equivalently, and perhaps  more conveniently,
\begin{equation}\label{FFma}
  \begin{split}
 \FF_m(z)
&
=  \sum_{k=0}^{m} \binom{m}{k} F_{k}(z)
=  \sum_{k=0}^{m-1} \binom{m}{k} F_{k}(z)
+z\Phi\bigpar{\FF_m(z)}.
\\&
=  \sum_{k=0}^{m-1}(-1)^{m-k+1} \binom{m}{k} \FF_{k}(z)
+z\Phi\bigpar{\FF_m(z)}.
  \end{split}
\end{equation}
For example,
for $m=2$,
\begin{equation}\label{F2a}
  \begin{split}
  F_2(z)
&=z\Phi\bigpar{F_2(z)+2F_1(z)+F_0(z)}
=z+z\bigpar{F_0(z)+2F_1(z)+F_2(z)}^2
\\&
=z+z\bigpar{F_0(z)+2F_1(z)}^2+2z\bigpar{F_0(z)+2F_1(z)}F_2(z)+zF_2(z)^2
,
  \end{split}
\end{equation}
and
\begin{equation}\label{FF2a}
  \begin{split}
\FF_2(z)= F_0(z)+2F_1(z)+z\Phi\bigpar{\FF_2(z)}   
= -\FF_0(z)+2\FF_1(z)+z+z\FF_2(z)^2   
.
  \end{split}
\end{equation}

Explicitly, we obtain from \eqref{F2a} or \eqref{FF2a}
\begin{multline}  \label{F2b}
F_2(z)=
\frac{1}
{2 z}
\left(
2 \sqrt{2 \sqrt{1-4 z^2}-1-4 z^2}
-    \sqrt{1-4 z^2}
\right.
    \\
-\left.
    \sqrt{4 \sqrt{2 \sqrt{1-4 z^2}-1-4 z^2}-2 \sqrt{1-4 z^2}-1-4 z^2}
\right)
.
\end{multline}

\subsection{Singularity analysis: full binary trees}
\label{sec:sing}

Let $\rho_m$ be the radius of convergence of $F_m(z)$; then $\rho_m$ is a
singularity of $F_m(z)$ (of square-root type).
We see from \eqref{F0x} that
\begin{equation}
  1-4\rho_0^2=0
  ,
\end{equation}
and thus 
\begin{equation}
  \rho_0=\frac12.
    \label{eq:rho:0}
\end{equation}
Since full binary trees can only have odd number of nodes, we have \(a_{2m} = 0\) for \(m \ge 0\).
For odd \(n\), applying singular analysis to \eqref{F0x} gives
\begin{align}
    a_{n}
    =
    \left( 
        1 + O\left( n^{-1} \right)
    \right)
    \lambda_{0}
    {n^{-\frac{3}{2}}\rho_{0}^{-n}}
    ,
    \label{eq:a:n}
\end{align}
where \(\lambda_{0}=\sqrt{\frac{2}{\pi}}\).
See  \cite[Theorem~VI.2]{Flajolet2009} for details.
%For details, see Chapter IV of \citet{Flajolet2009}.
(In fact, in this case we have the well-known exact formula
$a_{2m+1}=C_m:=(2m)!/(m!\,(m+1)!)$, the Catalan numbers \cite[p.~67]{Flajolet2009}.)

Similarly, \eqref{F1x} shows that
\begin{equation}
{2\sqrt{1-4\rho_1^2}-1-4\rho_1^2}=0
,
    \label{eq:rho:1}
\end{equation}
and thus
\begin{equation}\label{rho1a}
  \rho_1=\frac{\sqrt{2\sqrt3-3}}{2}
\doteq  0.340625
.
\end{equation}
Using the standard singular analysis recipe 
(see \cite[Figure VI.7, p.~394]{Flajolet2009}),
\begin{align}
    \ai_n = 
    \left( 
        1
        +
        O\left( 
            n^{-1}
        \right)
    \right)
    \lambda_{1}
    n^{-\frac{3}{2}}
    \rho_{1}^{-n}
    ,
    \label{eq:a:n:1}
\end{align}
where \(\lambda_{1} = \sqrt{ \frac{3+ \sqrt{3}}{ \pi} } \doteq1.227297\).
(Such computations can be partially automated with Maple, see, e.g.,
\cite{Salvy1991}.)
Thus \eqref{ER} implies that
\begin{align}
    \E R(\ctn) = 
    \left( 
        1 + O \left(n^{-1} \right)
    \right)
    \frac{\lambda_{1}}{\lambda_{0}}
    \left(
        \frac{
            \rho_{0}
        }{
            \rho_{1}
        }
    \right)^{n}
    .
    \label{eq:ER}
\end{align}

For the second moment,
\eqref{F2b} similarly yields
\begin{align}
    \rho_{2}
    =
    \frac{1}{2} \sqrt{2 \sqrt{48 \sqrt{2}+59}-8 \sqrt{2}-11}
    \doteq
    0.231676
    .
    \label{eq:rho:2}
\end{align}
Thus
\begin{align}
    \aii_{n}
    =
    \left( 
        1 + O \left(n^{-1} \right)
    \right)
    \lambda_{2}
    n^{-\frac{3}{2}}
    \rho_{2}^{-n}
    \label{eq:a:n:2}
    ,
\end{align}
where \(\lambda_{2}\doteq 1.883418\)
is a constant. Then by \eqref{ERm}
\begin{align}
    \E R(T_{n})^{2}
    =
    \left( 
        1 + O \left(n^{-1} \right)
    \right)
    \frac{\lambda_{2}}{\lambda_{0}}
    \left(
        \frac{\rho_{0}}{\rho_{1}}
    \right)^{n}
    .
    \label{eq:ER:2}
\end{align}

%A singularity analysis, as in your draft, yields
%asymptotics for moments of $R(\ctn)$.
%(We could do 3 or 4 moments explicitly or numerically).

It is not difficult to prove by induction that there exist sequences of numbers
\(\lambda_{m}>0\) and \(\rho_{0} > \rho_{1} > \cdots\) such that for every fixed \(m \ge 1\),
\begin{align}
    \ax{m}_{n} =
    \left( 1 + O\left( n^{-1} \right)\right)
    \lambda_{m} n^{-3/2} \rho_{m}^{-n}
    \label{eq:a:n:m}
\end{align}
and
\begin{align}
    \e{R(T_{n})^{m}} =
    \left( 1 + O\left( n^{-1} \right)\right)
    \frac{
        \lambda_{m} 
    }{
        \lambda_{0}
    }
    \left(
        \frac{
            \rho_{0}
        }{
            \rho_{m}
        }
    \right)^{n}
    .
    \label{eq:ER:m}
\end{align}
This is \eqref{eq:mmt3} with $\gam_m=\gl_m/\gl_0$ and
$\tau_m=\rho_0/\rho_m=(2\rho_m)\qw$.
In particular,
\begin{align}
    \tau_1
    &
    =\frac{1}{2\rho_1}=\frac{\sqrt{2\sqrt3+3}}{\sqrt3}
    =\sqrt{\frac{2}{\sqrt3}+1}
    \doteq 1.467890,
    %1.4678898250138705587
\label{bin:tau1}  
    \\
    \tau_{2}
    &
    =
    \frac{1}{2 \rho_{2}}
    =
    \frac{1}{7} \sqrt{57+40 \sqrt{2}+2 \sqrt{1635+1168 \sqrt{2}}}
    \doteq
    2.158182
    ,
\label{bin:tau2}
\end{align}
and
\begin{equation}
    \gamma_{1} 
    = 
    \sqrt{
        \frac{
            3+\sqrt{3}
        }{
            2
        }
    }
    \doteq
    1.538189,
    \qquad
    \gamma_{2}
    \doteq
    2.360501
    .
    \label{bin:gamma}
\end{equation}
We do not have a closed form of \(\rho_{m}\) or $\tau_m$ for \(m \ge 3\).
Table \ref{table:sing} gives the
numerical values of \(\tau_{m}\) and \(\rho_{m}\) for \(m\) up to \(10\).
\begin{table}[h!]
    \centering
    \begin{tabular}{|>{$}c<{$}  >{$}c<{$} | >{$}c<{$}  >{$}c<{$}|} 
        \hline
        \tau_1 & 1.467890 & \tau_6 &    10.22570 \\
        \tau_2 & 2.158182 & \tau_7 &    15.13130 \\
        \tau_3 & 3.177848 & \tau_8 &    22.41257 \\
        \tau_4 & 4.685754 & \tau_9 &    33.22804 \\
        \tau_5 & 6.918003 & \tau_{10} & 49.30410 \\
        \hline
    \end{tabular}
    \quad
    \begin{tabular}{|>{$}c<{$}  >{$}c<{$} | >{$}c<{$}  >{$}c<{$}|} 
        \hline
        \rho_1 & 0.340625 & \rho_6 & 0.048896 \\
        \rho_2 & 0.231676 & \rho_7 & 0.033044 \\
        \rho_3 & 0.157339 & \rho_8 & 0.022309 \\
        \rho_4 & 0.106706 & \rho_9 & 0.015048 \\
        \rho_5 & 0.072275 & \rho_{10} & 0.010141 \\
        \hline
    \end{tabular}
    \medskip
    \caption{Numerical values of \(\tau_{m}\) and \(\rho_{m}\) for full
      binary trees.}
    \label{table:sing}
\end{table}
\vspace{-5ex}
%\begin{table}[h!]
%    \centering
%    \begin{tabular}{|>{$}c<{$}  >{$}c<{$} | >{$}c<{$}  >{$}c<{$}|} 
%        \hline
%        \rho_0 & 0.500000 & \rho_5 & 0.072275 \\
%        \rho_1 & 0.340625 & \rho_6 & 0.048896 \\
%        \rho_2 & 0.231676 & \rho_7 & 0.033044 \\
%        \rho_3 & 0.157339 & \rho_8 & 0.022309 \\
%        \rho_4 & 0.106706 & \rho_9 & 0.015048 \\
%        \hline
%    \end{tabular}
%    \medskip
%    \caption{Numerical values of singularities}
%    \label{table:sing}
%\end{table}

\begin{remark}\label{Ralgebra}
It can be shown, using the equations above and taking resultants to
eliminate variables, that $\rho_1$, $\rho_2$ and $\rho_3$ are roots of the
equations
\begin{align}
    16\,\rho_1^{4}+24\,\rho_1^{2}-3&=0
    ,
    \label{P1}
    \\
    256\,\rho_2^{8}+2816\,\rho_2^{6}-32\,\rho_2^{4}+6384\,\rho_2^{2}-343
    &=0
    ,
    \label{P2}
    \\
    65536\,\rho_3^{16}+5111808\,\rho_3^{14}+70434816\,\rho_3^{12}
    -785866752\,\rho_3^{10}
    \quad&\notag
    \\+206968320\,\rho_3^{8}+10195628544\,\rho_3^{6}-16526908224\,\rho_3^{4}
    \quad&\notag
    \\+
    7520519520\,\rho_3^{2}-176201487
    &=0.
    \label{P3}
\end{align}
According to Maple, these polynomials are irreducible over the rationals; 
moreover, the polynomial in \eqref{P2} is irreducible over $\bbQ(\rho_1)$
and the polynomial in \eqref{P3} is irreducible over $\bbQ(\rho_1,\rho_2)$.
In particular, we have a strictly increasing sequence of fields 
$\bbQ\subset\bbQ(\rho_1)\subset\bbQ(\rho_1,\rho_2)\subset\bbQ(\rho_1,\rho_2,\rho_3)$.
We expect that this continues for larger $m$ as well, and that the fields
$\bbQ(\rho_1,\dots,\rho_m)$
form a strictly increasing sequence for $0\le m<\infty$.
\end{remark}

\begin{remark}\label{Rvar}
  The values in \eqref{bin:tau1}--\eqref{bin:tau2} show that $\tau_1^2<\tau_2$.
(In fact, $\tau_2/\tau_1^2\doteq 1.0016 $.) Hence \eqref{eq:mmt3} implies
that, as \ntoo,
\begin{equation}
\e{R(\ctn)^2} /\bigpar{\e{ R(\ctn)}}^2 \to \infty  
\end{equation}
and thus
\begin{equation}
  \Var[R(\ctn)] \sim \e{R(\ctn)^2}.
\end{equation}
We expect that the same holds for other  \cGWt{s}, but we have no
general proof.
\end{remark}

\begin{remark}\label{Rneq}
As said in \refR{Rmom}, it seems natural to combine Theorems \ref{TLN} and
\ref{Tmom} and guess that the moments of $R(\ctn)$ %given by \refT{Tmom}
asymptotically are as
the moments of the asymptotic log-normal distribution in \refT{TLN};
this means equality in \eqref{??}.
However, if equality holds in \eqref{??} for $m=1,2,3$, then
\begin{equation}
  \tau_1^3\tau_2^{-3}\tau_3=e^{(3-6+3)\mu+(3-12+9)\gss/2}=1,
\end{equation}
and thus 
\begin{equation}\label{*tau*}
\tau_3=\tau_2^3\tau_1^{-3}.  
\end{equation}
Equivalently,
$\rho_3=\rho_2^3\rho_1^{-3}\rho_0$.
However, in the case of full binary trees, 
we have noted in \refR{Ralgebra} that
$\rho_3\notin\bbQ(\rho_1,\rho_2)=\bbQ(\rho_0,\rho_1,\rho_2)$,  
so \eqref{*tau*} is impossible. In fact, a numerical calculation, using the
values in Table~\ref{table:sing}, yields in this case
\begin{equation}
\tau_3\tau_2^{-3}\tau_1^{3}  
=
\rho_3\qw\rho_2^{3}\rho_1^{-3}\rho_0
%    \Bigparfrac{\rho_0}{\rho_1}^{3}
%    \Bigparfrac{\rho_0}{\rho_2}^{-3}
%    \Bigparfrac{\rho_0}{\rho_3}
\doteq
0.99988.
\end{equation}
\end{remark}

\subsection{Proof of Theorems \ref{Tmom}--\ref{Tmom2}}
\label{sec:mmt:proof}

Consider a general \(\Phi(z)\) which satisfies the condition of \refT{Tmom2}.
We define the
\emph{weighted generating function} for \(m\)-partially colored trees by
\begin{equation}\label{wFm}
    \sF_m(z)
    :=
    \sum_{T}\sum_{T'_1,\dots,T_m'\subseteq_{r} T} w(T)z^{|T|}  
    = 
    \sum_{T}
    w(T)
    R(T)^m z^{|T|}  
    ,
\end{equation}
where \(w(T)\) is the weight of \(T\) defined in \refS{sec:simple}.
(Note that in case of full binary trees in \refSS{sec:gf}, 
$w(T)=1$ and \eqref{wFm} agrees with \eqref{Fm}.)
Then we have
\begin{equation}
    \label{eq:mmt}
    \E{R(\ctn)^{m}}
    =
    \frac{
        \sum_{T:|T|=n}
        w(T)
        R(T)^m
    }{
        \sum_{T:|T|=n}
        w(T)
    }
    =
    \frac{
        [z^{n}]
        \sF_{m}(z)
    }{
        [z^{n}]
        \sF_{0}(z)
    }
    .
\end{equation}

Following exactly the same argument as in \refSS{sec:gf}, 
we have a system
of equations
\begin{equation}\label{eqFms}
    \sF_{m}(z) = z \Phi\left( \sum_{k=0}^{m} \binom{m}{k} \sF_{k}(z) \right),
    \qquad
   m=0,1,\dots
    .
\end{equation}
By induction and the implicit function theorem
\cite[Theorem~B.4]{Flajolet2009}, there exist for each $m$ a
function \(\sF_{m}(z)\) that 
is analytic in some neighborhood of \(0\) (depending on $m$)
and satisfies \eqref{eqFms} there.
%Moreover, \(\sF_{m}(z)\) has only nonnegative
%coefficients and is thus strictly increasing for \(z>0\),
%with \(\sF_{m}(0)=0\).

For singularity analysis, we apply Theorem VII.3 of \cite{Flajolet2009}. We
need some preparations.
Define again $\FF_m(z)$ by \eqref{FFma}, and let
\begin{equation}\label{Hm}
H_m(z):=\FF_m(z)-F_m(z)=\sum_{k=0}^{m-1} \binom{m}{k} \sF_{k}(z)
,
\end{equation}
 and  
\begin{equation}
    \Psi_{m}(z,w) 
    \coloneqq
    z 
    \Phi
\bigpar{
        w +H_m(z)}
    .
    \label{gmh}
\end{equation}
Then the implicit equation \eqref{eqFms} can be written in the equivalent
forms
\begin{align}
  F_m(z)&=z\Phi\bigpar{\FF_m(z)}, \label{FF}
\\
  F_m(z)&=\Psi_m\bigpar{z,F_m(z)}.
\label{FG}  
\end{align}
Let $\rho_m>0$ be the radius of convergence of $F_m(z)$, and let
$s_m:=F_m(\rho_m)\le\infty$. 
We claim that $\infty>\rho_0>\rho_1>\dots$, and that for every $m$,
$s_m<\infty$ and
\begin{align}
%    \label{eqgam0}
%    \Psi_{m}(\rho_{m}, s_{m})& = s_{m}
%    ,
%\\
    \frac{\partial \Psi_{m}}{\partial w}
    \left(
    {\rho_{m}, s_{m}}
    \right)
   & =
    1
    .
\label{eqgam'}
\end{align}
We prove this claim by induction. 
(The base case $m=0$ is well-known, 
see \cite[Theorem VI.6, p.~404]{Flajolet2009}, and follows by minor
modifications of the argument below.)
Note first that, by \eqref{FF}, $\FF_m(z)\le R$ when $0<z<\rho_m$, 
and thus, letting $z\upto\rho_m$,
\begin{equation}\label{leR}
  s_m+H_m(\rho_m) =F_m(\rho_m)+H_m(\rho_m)=\FF_m(\rho_m)\le R.
\end{equation}

Next,
by \eqref{gmh},
\begin{equation}\label{gmh'}
    \frac{\partial \Psi_{m}}{\partial w}(z,w)=z\Phi'\bigpar{w+H_m(z)}
    ,
\end{equation}
and, in particular,
\begin{equation}\label{gmh'ff}
    \frac{\partial \Psi_{m}}{\partial w}\bigpar{z,F_m(z)}=z\Phi'\bigpar{\FF_m(z)}.
\end{equation}

Since 
\(\sF_{m}(z)\) has only nonnegative
coefficients, it has a singularity at $\rho_m$.
This singularity can arise in one of three ways:
\begin{romenumerate}
\item \label{xa}
$\rho_m\ge\rho_{m-1}$. %($m\ge1$)
%(Then $\Psi_m(z,w)$ is not analytic at $z=\rho_m$, because $H_m(z)$ is not.)

\item \label{xb}
$\FF_m(\rho_m)=s_m+H_m(\rho_m)= R$. %\le\infty$. 
(Recall \eqref{leR}.)

\item \label{xc}
\eqref{eqgam'} holds.
\end{romenumerate}
In fact, if neither \ref{xa} nor \ref{xb} holds, then $\rho_m<\infty$,
$s_m<\infty$ and 
$\Psi_m$ is analytic in a neighbourhood of $(\rho_m,s_m)$.
If also \ref{xc} does not hold, then $F_m(z)$ is analytic in a
neighbourhood of $\rho_m$ by \eqref{FG} and
the implicit function theorem, 
which
contradicts that \(F_{m}(z)\) has a singularity  at \(\rho_{m}\).

We will show that \ref{xa} and \ref{xb} are impossible; thus \ref{xc} is the
only possibility.

Differentiating \eqref{FG}, we obtain
\begin{equation}\label{fm'}
  F_m'(z)=\frac{\partial \Psi_m}{\partial z}\bigpar{z,F_m(z)}
+ \frac{\partial \Psi_m}{\partial w}\bigpar{z,F_m(z)} F_m'(z).
\end{equation}
For $0<z<\rho_m$, all terms in \eqref{fm'} are positive and finite; hence
$  F_m'(z)> \frac{\partial \Psi_m}{\partial w}\bigpar{z,F_m(z)} F_m'(z)$
and
\begin{equation}\label{d<1}
 \frac{\partial \Psi_m}{\partial w}\bigpar{z,F_m(z)}<1,
\qquad 0<z<\rho_m.
\end{equation}

Suppose now that \ref{xa} holds. Then $F_m(z)$ is analytic for
$|z|<\rho_{m-1}$.
Furthermore, by induction, $F_{m-1}(\rho_{m-1})=s_{m-1}<\infty$, and
$H_{m-1}(\rho_{m-1})<\infty$.
Hence,
$\FF_{m-1}(\rho_{m-1})=F_{m-1}(\rho_{m-1})+H_{m-1}(\rho_{m-1})<\infty$.
This and  the definition \eqref{FFma} yield
$\FF_m(\rho_{m-1})>\FF_{m-1}(\rho_{m-1})$,
and thus, using \eqref{gmh'ff},
\begin{equation}\label{sno}
  \begin{split}
&    \lim_{z\upto\rho_{m-1}}
  \frac{\partial \Psi_m}{\partial w}\bigpar{z,F_m(z)}
=\rho_{m-1}\Phi'\bigpar{\FF_m(\rho_{m-1})}
\\& \qquad
>\rho_{m-1}\Phi'\bigpar{\FF_{m-1}(\rho_{m-1})}
=  \frac{\partial \Psi_{m-1}}{\partial w}\bigpar{\rho_{m-1},F_{m-1}(\rho_{m-1})}
=1,
  \end{split}
\end{equation}
by the induction hypothesis \eqref{eqgam'} for $m-1$.
However, \eqref{sno} contradicts \eqref{d<1}. Hence, \ref{xa} cannot hold,
and $\rho_m<\rho_{m-1}$.

Next, for $0<z<\rho_m$, by \eqref{FF}, \eqref{gmh'ff} and \eqref{d<1},
\begin{equation}\label{gagnef}
  \begin{split}
\frac{\FF_m(z)\Phi'(\FF_m(z))}{\Phi(\FF_m(z))}    
=
\frac{z\FF_m(z)\Phi'(\FF_m(z))}{F_m(z)}    
=
\frac{\FF_m(z)}{F_m(z)}     \frac{\partial \Psi_m}{\partial w}\bigpar{z,F_m(z)}
<
\frac{\FF_m(z)}{F_m(z)}
%=1+\frac{H_m(z)}{F_m(z)}
.
  \end{split}
\end{equation}
Since $F_k(z)\le F_m(z)$ when $0\le k\le m$ by \eqref{wFm},
the \rhs{} of \eqref{gagnef} is by \eqref{FFma} bounded by $2^m$.

Suppose now that \ref{xb} holds.
Then, as $z\upto\rho_m$, 
\begin{equation}
  \label{bisp}
\FF_m(z)\to \FF_m(\rho_m)=R,   
\end{equation}
and thus
\eqref{gagnef} implies
\begin{equation}\label{djurmo}
  \begin{split}
\lim_{\zeta\upto R}
\frac{\zeta\Phi'(\zeta)}{\Phi(\zeta)}    
\le
\limsup_{z\upto \rho_m}
\frac{\FF_m(z)}{F_m(z)}
\le 2^m.
  \end{split}
\end{equation}
Consider now two cases. First, if $\Phi(R)<\infty$, then the \lhs{} of
\eqref{djurmo} is $R\Phi'(R)/\Phi(R)=\infty$
by the assumption \eqref{philmx}, %$\Phi'(R)=\infty$, 
which is a contradiction.
On the other hand, if $\Phi(R)=\infty$, then \eqref{FF} and \eqref{bisp}
yield
\begin{equation}\label{tuna}
%  F_m(\rho_m)=
\lim_{z\upto\rho_m} F_m(z) 
=\lim_{z\upto\rho_m}z\Phi\bigpar{\FF_m(z)} 
=\rho_m \Phi(R)=\infty.
\end{equation}
We have shown that $\rho_m<\rho_{m-1}\le\rho_k$ for every $k < m$, and thus
\eqref{Hm} shows that $H_m$ is analytic at $\rho_m$, and
$H_m(\rho_m)<\infty$. Hence, in this case \eqref{djurmo} yields, using
\eqref{tuna},
\begin{equation}\label{djuras}
  \begin{split}
\lim_{\zeta\upto R}
\frac{\zeta\Phi'(\zeta)}{\Phi(\zeta)}    
\le
\limsup_{z\upto \rho_m}
\frac{F_m(z)+H_m(z)}{F_m(z)}
= 1+\limsup_{z\upto \rho_m}\frac{H_m(z)}{F_m(z)}
=1,
  \end{split}
\end{equation}
which contradicts the assumption \eqref{philm}.
We have thus reached a contradiction in both cases, 
which shows that \ref{xb} cannot hold, so
\begin{equation}
  \label{zebra}
\FF_m(\rho_m)=s_m+H_m(\rho_m)<R.
\end{equation}

Hence, \ref{xc} holds. Furthermore, by \eqref{zebra}, $s_m<\infty$, and letting
$z\upto\rho_m$ in \eqref{FG} yields
\begin{equation}
      \label{eqgam0}
    \Psi_{m}(\rho_{m}, s_{m})= s_{m}
    .
\end{equation}

We now apply  \cite[Theorem VII.3, p.~468]{Flajolet2009}, noting that
the conditions are satisfied by the results above, in particular
\eqref{eqgam0}, \eqref{eqgam'} and \eqref{zebra}.
This theorem shows that
\(\sF_{m}(z)\) has a square-root singularity at \(\rho_{m}\), and that
its coefficients satisfy
\begin{equation}
    [z^{n}] \sF_{m}(z)
    =
    \frac{
        \lambda_{m}
    }{
        \sqrt{n^{3}}{}
    }
        \rho_{m}^{-n}
    \left( 
        1
        +
        O({n^{-1}})
    \right)
    ,
    \label{sFco}
\end{equation}
where \(\lambda_{m}>0\) is a constant.
(In the periodic case, as usual we consider only $n$ such that $\cT_n$ exists.)
It follows from \eqref{eq:mmt} that
\begin{equation}
    \label{eq:mmt1}
    \E{R(\ctn)^{m}}
    =
    \frac{
        [z^{n}]
        \sF_{m}(z)
    }{
        [z^{n}]
        \sF_{0}(z)
    }
    =
    \frac{
        \lambda_{m}
    }{
        \lambda_{0}
    }
    \left( 
        \frac{
            \rho_{0}
        }{
            \rho_{m}
        }
    \right)^{m}
    \left( 
        1
        +
        O
        \left( 
            n^{-1}
        \right)
    \right)
    .
\end{equation}
Letting \(\gamma_{m}=\lambda_{m}/\lambda_{0}\) and
\(\tau_{m}=\rho_{0}/\rho_{m}\), we have shown \eqref{eq:mmt3}.

This prove \refT{Tmom2}, and thus also the equivalent \refT{Tmom}.
\qed

\subsection{A counter example}\label{SScounter}

The following example shows that \refT{Tmom2} does not hold without the
condition \eqref{philmx}.

\begin{example}\label{Ecounter}
  Take the generator
  \begin{equation}\label{fia}
\Phi(z)=   \Phi_a(z)=a+\frac{1-a}{\zeta(4)}\sumk \frac{z^k}{k^4},
  \end{equation}
where $0<a<a_0:=1-\zeta(4)/\zeta(3)$.
Then $R=1$, $\Phi(R)=1$ and
\begin{equation}\label{nu=}
  \nu 
:=\lim_{z \upto R} \frac{z \Phi'(z)}{\Phi(z)}
= \Phi'(1)=(1-a)\frac{\zeta(3)}{\zeta(4)}=\frac{1-a}{1-a_0}>1,
\end{equation}
so \eqref{philm} holds.
%$F_0(z)$ has a singularity of square root type at $\rho_0<1$.

Suppose now that there exists $\rho_1<1$ such that $s_1:=F_1(\rho_1)<\infty$
and $\frac{\partial \Psi_1}{\partial w}(\rho_1,s_1)=1$, and thus,
see \eqref{gmh'ff},
\begin{equation}\label{r9}
  \rho_1 \Phi'\bigpar{F_0(\rho_1)+F_1(\rho_1)}=1.
\end{equation}
Then $F_0(\rho_1)+F_1(\rho_1)\le R=1$. Since $F_0(z)\le F_1(z)$ for every
$z\ge0$, this implies
\begin{equation}\label{trna}
F_0(\rho_1)\le \tfrac12.
\end{equation}
On the other hand, 
$\Phi'\bigpar{F_0(\rho_1)+F_1(\rho_1)}\le \Phi'(1)=\nu$, and thus
\eqref{r9} implies $\rho_1\ge\nu\qw$. 
Furthermore, $F_0(z)=z\Phi\bigpar{F_0(z)}$. Thus, if $x:=F_0(\rho_1)\le\frac12$,
we have $x=\rho_1 \Phi(x)$, and thus
\begin{equation}
  x=\rho_1\Phi(x)\ge\nu\qw\Phi(x)
  ,
\end{equation}
which yields, recalling \eqref{nu=},
\begin{equation}\label{mrna}
\Phi(x)=  \Phi_a(x)\le \nu x = \frac{1-a}{1-a_0}x.
\end{equation}
We claim that this is impossible if $a$ is close to $a_0$. 
In fact, suppose that for every $a<a_0$ there exists $x=x_a\le\frac12$ such
that \eqref{mrna} holds. Then, by compactness, we may take a sequence
$a_n\upto a_0$ such that
$x_{a_n}$ converges to some $x_*\in[0,\frac12]$, and then \eqref{mrna} implies
\begin{equation}\label{rrna}
 \Phi_{a_0}(x_*)\le  x_*,
\end{equation}
which is a contradiction since $\Phi_{a_0}(1)=1$ and
$\Phi'_{a_0}(x)<\Phi'_{a_0}(1)=1$ for $x_*<x<1$.

Consequently, we can find $a<a_0$ such that the simply generated tree with
generator \eqref{fia} does not have $F_1$ with a singularity of the type
above.
Hence, in this case, $\rho_1$ is instead given by
\ref{xb} in \refSS{sec:mmt:proof}, i.e.,
$\FF_1(\rho_1)=1$,
which by \eqref{FF}
implies
\begin{align}
F_1(\rho_1)&=\rho_1\Phi\bigpar{\FF_1(\rho_1)}=\rho_1,
\\
F_0(\rho_1)&=\FF_1(\rho_1)-F_1(\rho_1)=1-\rho_1.
\end{align}
We have shown that
$\frac{\partial \Psi_1}{\partial w}\bigpar{\rho_1,F_1(\rho_1)}<1$, and thus
it follows from \eqref{fm'} and $\Phi'(1)<\infty$
that $\lim_{z\upto \rho_1}F_1'(z)<\infty$.
Hence the singularity of $F_1$ at $\rho_1$ is not of  square root type,
and  the asymptotic formula
\eqref{eq:mmt3} cannot hold.

We leave it as an open problem to find the asymptotics of $\E R(\ctn)$ and
higher moments in this case.
\end{example}

\section{General subtrees}
\label{sec:gen}

We have in \refSs{sec:mmt} considered root subtrees.
Estimates for general non-fringe subtrees follow from \eqref{SR2},
but more precise results can be obtained
by introducing the corresponding generating functions
\begin{equation}\label{wGm}
    G_m(z)
    :=
    \sum_{T} w(T)S(T)^m z^{|T|}
=  \sum_{T}\sum_{T'_1,\dots,T_m'\subseteq T} w(T)z^{|T|}
,
\end{equation}
\cf{} \eqref{wFm} and note that $G_0(z)=F_0(z)$.

For simplicity, we study first the case $m=1$ in detail, and
as in \refS{sec:mmt}, we consider first the example of full binary
trees.
We assume throughout this section that the  assumptions of \refT{Tmom2} hold.

\subsection{The mean, full binary trees}
Let \(\cG_{1}\) be the combinatorial class of
pairs of trees \( (T, T')\) such that \(T'\) is subtree of \(T \in
\cF_{0}\). In other words, an
element of \(\cG_{1}\) is a full binary tree with one non-fringe subtree
colored.  Such a partially
colored tree is either a full binary with a root subtree colored (an element of \(\cF_{1}\)), or a uncolored root
together with a left (right) uncolored subtree (an element of \(\cF_{0}\)) and a partially colored
right (left) subtree (an element of \(\cG_{1}\)).
Thus \(\cG_{1}\) has the specification
\begin{equation}
    \label{eq:G1}
    \cG_{1} =
    \cF_{1}
    +
    \cZ \times \cF_{0} \times \cG_{1}
    +
    \cZ \times \cG_{1} \times \cF_{0}
    .
\end{equation}
Therefore, \(G_{1}(z)\), the generating function of \(\cG_{1}\)
given by
%\cf{} \eqref{wGm},
\begin{equation}
    G_1(z)
    \coloneqq
    \sum_{T} \sum_{T' \subseteq T}
    z^{|T|}
    =
    \sum_T S(T)z^{|T|}
    =: \sumn \bi_n z^n
    ,
\end{equation}
satisfies
\begin{equation}\label{G1a}
    G_1(z)=F_1(z)+2zF_0(z)G_1(z)
    .
\end{equation}
Thus
\begin{equation}\label{G1b}
    G_1(z)=\frac{F_1(z)}{1-2zF_0(z)}.
\end{equation}

By \eqref{G1b}, $G_1(z)$ has the same
radius of convergence $\rho_1$ as
\(F_{1}(z)\), with singularities at the same points \(\pm \rho_{1}\).
(It is easily verified that the denominator $1-2zF_0(z)\neq0$ for
$|z|\le\rho_1$, see also \eqref{nora}.)
Since the singular expansions of the denominator at \(\pm \rho_{1}\) are simply
both \(1-2 \rho_{1} F_{0}(\rho_{1})\), we obtain from \eqref{G1b}
by singularity analysis
\begin{equation}
    \bi_n \sim (1-2\rho_1 F_0(\rho_1))\qw\ai_n
    ,
\end{equation}
as $n\to\infty$. Thus, using \eqref{F0x} and \eqref{rho1a},
\begin{equation}
    \label{eq:ab}
    \ai_n/  \bi_n
    \to 1-2\rho_1 F_0(\rho_1)
    =\sqrt3-1
\doteq0.732.
\end{equation}
Therefore the root subtrees form the majority of all subtrees.

\subsection{The mean, general trees}

For simply generated trees with the generator \(\Phi(x)\),
\eqref{eq:G1} and \eqref{G1a} can be generalized to
\begin{equation}
    \cG_{1} = \cF_{1} + \cZ \times \cG_{1} \times \Phi'(\cF_{0}),
    \label{eq:G1:1}
\end{equation}
and
\begin{equation}
    G_{1}(z)
    =
    F_{1}(z)
    +
    z
    G_{1}
    (z)
    \Phi'(
    F_{0}(z)
    )
    .
    \label{eq:G1:3}
\end{equation}
Therefore,
\begin{equation}\label{G1c}
    G_1(z)=\frac{F_1(z)}{1-z\Phi'(F_0(z))}.
\end{equation}
Note that for any $m\ge0$ and $0<z<\rho_m$,
by  \eqref{gmh'ff} and \eqref{eqgam'},
  \begin{equation}\label{nora}
    z \Phi'\bigpar{\cFF_m(z)}
<
    \rho_m \Phi'\bigpar{\cFF_m(\rho_m)}
= \frac{\partial \Psi_{m}}{\partial w}\left({\rho_{m}, s_{m}}\right)
=1.
  \end{equation}
Together with \(F_{0}(z) = \cFF_{0}(z)\), 
this shows that
the denominator in \eqref{G1c} is non-zero for $|z|<\rho_0$, and in
particular for $|z|\le\rho_1$.
Thus
$G_1(z)$ has the same dominant singularities with $|z|=\rho_1$ as
$F_1(z)$, and it follows that, as \ntoo,
\begin{equation}\label{ts2}
  \frac{\E R(\ctn)}{\E S(\ctn)}
=\frac{\ai_n}{\bi_n}
\to 1-\rho_1 \Phi'(F_0(\rho_1))>0.
\end{equation}
\begin{remark}
  \label{Rperiodic}
In the periodic case, when $F_1(z)$ has
$k\ge1$ singularities on the circle $|z|=\rho_1$, it is easily verified that
$1-z \Phi'\bigpar{\FF_1(z)}$ is a power series in $z^k$ and thus
has the same value at all these singularities,
\cf{} the full binary case above.
\end{remark}

\subsection{Higher moments}
The generating functions $G_m$
for higher moments of $S(\ctn)$ can be found
recursively by similar methods. The recursion becomes more complicated than
for $F_m$, however.
We introduce the generating functions for mixed moments of the numbers of
root subtrees and general subtrees
\begin{equation}
  G_{m,\ell}
:=\sum_{T} w(T)S(T)^m R(T)^\ell z^{|T|}
=  \sum_{T}\sum_{\substack{T'_1,\dots,T_m'\subseteq T,\\ T''_1,\dots,T_\ell''\subseteq_r T} }
w(T)z^{|T|}
.
\end{equation}
Note that $G_{m,0}=G_m$ and $G_{0,\ell}=F_\ell$.
It follows from \eqref{SR2} (or the recursions below)
that $G_{m,\ell}(z)$ has the same radius of
convergence $\rho_{m+\ell}$ as $F_{m+\ell}(z)$.

Consider first, as examples, the cases with $m+\ell=2$.
$G_{1,1}(z)$ is the generating function of the combinatorial class
$\cG_{1,1}$
consisting of triples $(T,T',T'')$ where $T'$ is a subtree and $T''$ a root
subtree of $T$, counted with weights $w(T)$ determined by $\Phi(z)$ by
\eqref{weight} and \eqref{gen:phi}.
%Regard the general subtree $T'$ as green and
%the root subtree $T''$ as red.

Let $(T,T',T'')\in\cG_{1,1}$.
Denote the children of the root $o\in T$ by $v_1,\dots,v_D$, and let
$T_{1},\dots,T_{D}$ be the corresponding fringe subtrees of $T$.
The subtree $T'$ is either a root subtree, and then $(T,T',T'')\in\cF_2$,
or it is a subtree of one of the fringe trees $T_{j}$.
Furthermore, the root subtree $T''$ is determined by choosing for each
fringe subtree $T_{j}$ either  a root subtree or nothing.
Hence, in the case $(T,T',T'')\notin\cF_{2}$,
for some $j_0\le D$, we choose either
$(T_{j_0},T'_{j_0},T''_{j_0})\in \cG_{1,1}$ or
$(T_{j_0},T'_{j_0})\in \cG_{1,0}$; at the same time, we choose for each $j\neq
j_0$ either $(T_j,T''_j)\in \cG_{0,1}=\cF_1$ or just $T_j\in\cG_{0,0}=\cF_0$.
Consequently,
\begin{equation}\label{GQ}
  \cG_{1,1}=\cF_2+ \cZ\times\sum_{D=1}^\infty w_D D \bigpar{\cG_{1,1}+\cG_{1,0}}\times
\bigpar{\cF_1+\cF_0}^{D-1}
,
\end{equation}
and thus
\begin{equation}\label{G11}
  G_{1,1}(z)=
  F_2(z)+ 
  z \bigpar{G_{1,1}(z)+G_{1,0}(z)}\Phi'\bigpar{\FF_1(z)}
  .
\end{equation}
Consequently,
\begin{equation}\label{G11+}
  G_{1,1}(z)=\frac{F_2(z)+ zG_{1,0}(z)\Phi'\bigpar{\FF_1(z)}}
  {1-z\Phi'\bigpar{\FF_1(z)}}
  .
\end{equation}

Similarly, $\cG_{2,0}$ is the class of triples $(T,T''_1,T''_2)$ where both
$T''_1$ and $T''_2$ are general subtrees of $T$. The case when both $T''_1$
and $T''_2$ are root trees gives $\cF_2$, and the case where, say,
$T''_1$, is a root tree but $T''_2$ is not gives $\cG_{1,1}\setminus\cF_2$,
found above.
Finally, if neither $T''_1$ nor $T''_2$ is a root tree, then they are
determined by one subtree in a fringe tree $T_{j_1}$ and one subtree in $T_{j_2}$,
where $j_1$ and $j_2$ may be equal or not.
This leads to, arguing as in \eqref{GQ}--\eqref{G11},
\begin{equation}\label{G20}
  \begin{split}
        G_{2,0}(z) 
        &
        =
        F_2(z)+2\bigpar{G_{1,1}(z)-F_2(z)}
        \\
        &\qquad
        +zG_{2,0}(z)\Phi'\bigpar{F_0(z)}
        +zG_{1,0}(z)^2\Phi''\bigpar{F_0(z)},
  \end{split}
\end{equation}
and thus
\begin{equation}\label{G20+}
  G_{2,0}(z)=
\frac{2G_{1,1}(z)-F_2(z)+zG_{1,0}(z)^2\Phi''\bigpar{F_0(z)}}
{1-z\Phi'\bigpar{F_0(z)}}.
\end{equation}

Singularity analysis of
\eqref{G11+} and \eqref{G20+} show that
\begin{align}
  \frac{\E[S(\ctn)R(\ctn)]}{\E[R(\ctn)^2]}
&=\frac{[z^n] G_{1,1}(z)}{[z^n]F_2(z)}
\to 
\frac{1}{1-\rho_2\Phi'(\FF_1(\rho_2))},
\label{alpha11}
\\
  \frac{\E[S(\ctn)^2]}{\E[R(\ctn)^2]}
&=\frac{[z^n] G_{2,0}(z)}{[z^n]F_2(z)}
\to 
\frac{1}{1-\rho_2\Phi'(F_0(\rho_2))}
\Bigpar{\frac{2}{1-\rho_2\Phi'(\FF_1(\rho_2))}-1},
\label{alpha20}
\end{align}
where all denominators are positive by \eqref{nora}.

The argument is easily extended to higher powers, and in principle can any
$G_{m,\ell}(z)$  be found recursively by this method; however, the formulas
will be more and more complicated, and we see no simple general formula.
On the other hand, we are really only interested in the singular parts, and
thus we can ignore most terms.

\begin{lemma}\label{LS}
For each $m,\ell\ge0$,
  there exist functions $\gf_{m,\ell}$ and $\psi_{m,\ell}$ that are analytic
 for    \(|z|<\rho_{m+\ell-1}\)
%  in a neighborhood of $|z|\le\rho_{m+\ell}$ 
such that 
  \begin{equation}
    G_{m,\ell}(z)=\gf_{m,\ell}(z)F_{m+\ell}(z)+\psi_{m,\ell}(z).
  \end{equation}
Furthermore,
$\ga_{m,\ell}:=\gf_{m,\ell}(\rho_{m+\ell})$ satisfies the recursion
$\ga_{0,\ell}=1$ and, for $m\ge1$, 
\begin{align}\label{ga}
  \ga_{m,\ell}
=
\frac{1
+ \rho_{m+\ell} \sum_{k=1}^{m-1}\binom{m}{k}\ga_{k,m+\ell-k}
\Phi'\bigpar{\FF_{m+\ell-k}(\rho_{m+\ell})}}
{1-\rho_{m+\ell}\Phi'\bigpar{\FF_\ell(\rho_{m+\ell})}}
.
\end{align}
\end{lemma}
\begin{proof}
The case $m=0$ is trivial, with $\gf_{0,\ell}(z)=1$ and
$\psi_{0,\ell}(z)=0$.
Thus, let $m\ge1$.
The combinatorial class $\cG_{m,\ell}$ consists of sequences
\begin{equation}
(T,T'_1,\dots,T'_m,T''_1,\dots,T''_\ell),
\end{equation}
where $T$ is a tree, counted with
weight $w(T)$, each $T'_i$ is a subtree and each $T''_j$ is a root subtree.
Let $k$ be the number of $T'_1,\dots,T'_m$ that are not root subtrees.
The case $k=0$ gives $\cF_{m+\ell}$.
Suppose $1\le k\le m$.
Let again  $T_1,\dots,T_D$ be the fringe trees rooted at the children of
the root $o$ of $T$.
Further suppose that the \(k\) non-root subtrees go into \(p \ge 1\) of
 \(T_{1}, \cdots, T_{D}\), 
    which we call \(T_{j_1},\cdots,T_{j_p}\). 
Then there are \(\binom{D}{p}\) ways to select
    \(T_{j_{1}},\cdots,T_{j_{p}}\).  
Suppose further that \(k_{i}\) of the \(k\) non-root subtrees
    go into \(T_{j_{i}}\). 
Given \(p\) positive integers \(k_{1},\cdots,k_{p}\) with
    \(k_{1}+\cdots+k_{p}=k\),
    there are
    \(\binom{k}{k_{1},\cdots,k_{p}}\) ways of choosing how the \(k\)
    non-root subtrees are divided among \(T_{j_{i}},\cdots,T_{j_{p}}\). 
While \(T_{j_{r}}\) contains \(k_{r}\)
    marked general subtrees, it also contains  
\(i\le m+\ell-k\) marked root subtrees, 
which can be chosen in \( \binom{m+\ell-k}{i}\) ways.
    For any \(T_{j}\) such that \(j \notin \{j_{1},\cdots,j_{p}\}\), it too
contains up to
    \(m+\ell-k\) marked root subtrees. Hence, fixing \(k\), \(p\),
    \(k_{1},\cdots,k_{p}\), \(D\), we have the following term that contributes to
    \(G_{m,\ell}(z)\)
\begin{align}\label{sala:1}
z  w_{D} \binom{D}{p} \binom{k}{k_1,\dots,k_p}\prod_{j=1}^p
  \lrpar{\sum_{i=0}^{m+\ell-k}\binom{m+\ell-k}{i}G_{k_j,i}(z)}
\bigpar{\FF_{m+\ell-k}(z)}^{D-p}.
\end{align}
Summing over \(D \ge 1\) gives
\begin{align}\label{sala}
z  \frac{1}{p!}\binom{k}{k_1,\dots,k_p}\prod_{j=1}^p
  \lrpar{\sum_{i=0}^{m+\ell-k}\binom{m+\ell-k}{i}G_{k_j,i}(z)}
\Phi^{(p)}\bigpar{\FF_{m+\ell-k}(z)}.
\end{align}
In \eqref{sala}, $m+\ell-k\le m+\ell-1$, and thus 
$\Phi^{(p)}\bigpar{\FF_{m+\ell-k}(z)}$ has radius of convergence at least
$\rho_{m+\ell-1}$.
Similarly, $k_j+i\le m+\ell$ with equality only if $k_j=k$, and thus
$j=p=1$,
and also $i=m+\ell-k$; hence, except in the latter case,
\eqref{sala} has radius of convergence $\ge\rho_{m+\ell-1}$.
Consequently,
collecting all terms and lumping most of them together,
taking into account that there are \(\binom{m}{k}\) ways to choose the \(k\)
non-root     subtrees among \(T_{1}',\cdots,T_{m}'\),
\begin{align}
    G_{m,\ell}(z)
    =F_{m+\ell}(z)
    + z \sum_{k=1}^m \binom{m}{k}G_{k,m+\ell-k}(z)\Phi'\bigpar{\FF_{m+\ell-k}(z)}
    + \tphi(z),
    \label{Gml}
\end{align}
where $\tphi(z)$ has radius of convergence $\ge\rho_{m+\ell-1}$.
Hence, for any $m\ge1$,
\begin{align}
  G_{m,\ell}(z)
=
\frac{F_{m+\ell}(z)
+ z \sum_{k=1}^{m-1}\binom{m}{k}G_{k,m+\ell-k}(z)\Phi'\bigpar{\FF_{m+\ell-k}(z)}
+\tphi(z)}
{1-z\Phi'\bigpar{\FF_\ell(z)}},
\end{align}
where the denominator is non-zero for $|z|<\rho_{m+\ell-1}\le\rho_{\ell}$ 
by \eqref{nora}.
The results follow by induction in $m$.
\end{proof}

\begin{example}
    As an example of the recursion \eqref{Gml}, note that by \eqref{G11} and \eqref{G20},
    \begin{equation}\label{G20:1}
        \begin{split}
            G_{2,0}(z) 
            &
            =
            F_2(z)+
            %2 z \bigpar{G_{1,1}(z)+G_{1,0}(z)}\Phi'\bigpar{\FF_1(z)}
            2 z G_{1,1}(z)\Phi'\bigpar{\FF_1(z)}
            +2 z G_{1,0}(z)\Phi'\bigpar{\FF_1(z)}
            \\
            &\qquad
            +zG_{2,0}(z)\Phi'\bigpar{F_0(z)}
            +zG_{1,0}(z)^2\Phi''\bigpar{F_0(z)},
        \end{split}
    \end{equation}
where we collect the terms containing $G_{1,0}(z)$ into $\tphi(z)$ in
\eqref{Gml}.
\end{example}

\begin{proof}[Proof of \refT{TS}]
%  \refT{TS} is
An immediate consequence of \refL{LS}
and singularity analysis, 
arguing as at the end of the proof of \refT{Tmom2} in \refS{sec:mmt:proof},
\cf{} \eqref{sFco}--\eqref{eq:mmt1}.
Note that this yields
\begin{align}\label{gamma'}
\gamma'\ellm=\ga_{m,\ell}\gamma_{\ell+m}
,
\end{align}
with $\ga_{m,\ell}$ given by the recursion \eqref{ga}.
\end{proof}

\subsection{A numerical example}

It follows either from \refL{LS} or the direct computations in 
\eqref{G1c}, \eqref{alpha11}, \eqref{alpha20} that
\begin{align}
    \alpha_{1,0} 
    &
    =
    \frac{1}{1-\rho_1\Phi'(F_0(\rho_1))}
    ,
    \\
    \alpha_{1,1} 
    &
    =
    \frac{1}{1-\rho_2\Phi'(\FF_1(\rho_2))},
    \\
    \alpha_{2,0} 
    &
    =
    \frac{1}{1-\rho_2\Phi'(F_0(\rho_2))}
    \Bigpar{\frac{2}{1-\rho_2\Phi'(\FF_1(\rho_2))}-1}.
    \label{alpha20'}
\end{align}
Returning again to the full binary trees, we find from
\eqref{gamma'}--\eqref{alpha20'} by Maple 
\begin{align}
    \alpha_{0,1} &\doteq 1.366025,
    &
    \alpha_{1,1} &\doteq 1.339117,
    &
    \alpha_{2,0} &\doteq 1.893755,
    \label{alpha:bin}
    \\
    \gamma_{0,1}' &\doteq 2.101204,
    &
    \gamma_{1,1}' &\doteq 3.160952,
    &
    \gamma_{2,0}' &\doteq 4.470213.
\end{align}
Hence,  recalling \refR{Rvar},
we find the correlation coefficient between the numbers of root
subtrees and general subtrees
\begin{multline}
\frac{\Cov(S(\ctn), R(\ctn))}{\sqrt{\Var{S(\ctn)}\Var{R(\ctn)}}}
\sim
\frac{\e{S(\ctn) R(\ctn)}}{\sqrt{\e{S(\ctn)^2}\e{R(\ctn)^2}}}
\\
\sim 
\frac{\gamma'_{1,1}\tau_2^n}{\sqrt{\gamma'_{2,0}\tau_2^n\cdot\gamma_2\tau_2^n}}    
=\frac{\gamma'_{1,1}}{\sqrt{\gamma'_{2,0}\gamma_2}}
=\frac{\alpha_{1,1}}{\sqrt{\alpha_{2,0}}}
\doteq 0.973087.     
\end{multline}
Therefore, as might be expected, we have a strong but not perfect correlation.

\section{Average size of root subtrees}
\label{sec:size}

 Define
\begin{equation}\label{Fw1}
F_1(z,u):= \sum_{T}\sum_{T'\subseteq_{r} T} z^{|T|}  u^{|T'|}
=:\sumn \ai_n(u) z^n.
\end{equation}
Then $F_1(z)=F_1(z,1)$ and thus $\ai_n=\ai_n(1)$.
Moreover $\ai_n(u)/\ai_{n}$ is the probability generating function of \(X_{n}\),
where \(X_{n} = |T'|\) for a pair \( (T,T')\) where
$T'$ is a  root subtree of $T$,
chosen uniformly at random 
from all such pairs in \(\cF_{1}\) with \(|T|=n\).

$F_1(z,u)$ can be computed in the same way as in \refS{sec:mmt}.

\begin{example}
In the case of full binary trees, we obtain
\begin{equation}\label{Fu1a}
  \begin{split}
  F_1(z,u)
&=zu\Phi\bigpar{F_1(z,u)+F_0(z)}
=zu+zu\bigpar{F_0(z)+F_1(z,u)}^2
\\&
=uF_0(z)+2zuF_0(z)F_1(z,u)+zuF_1(z,u)^2,
  \end{split}
\end{equation}
with the explicit solution
\begin{equation}\label{Fu1x}
  \begin{split}
      F_1(z,u)&=\frac{1-\sqrt{1-4zu(zu+F_0(z))}}{2zu}-F_0(z)
\\&
=\frac{1-\sqrt{2u\sqrt{1-4z^2}+1-2u-4z^2u^2}}{2zu}-F_0(z).
  \end{split}
\end{equation}

By the well-known relationship between derivatives of probability generating
functions and factorial moments \cite[p.~158]{Flajolet2009}, we have
\begin{equation}\label{fact}
    \E{(X_{n})_{r}}
    =
    \left.
    \frac{d^{r}}{d u^{r}}
    \left( 
    \frac{\ai_{n}(u)}{\ai_{n}}
    \right)
    \right|_{u=1}
    =
    \frac{
        [z^{n}]
        \left.
            {\partial^{r}F_{1}(z,u)}/{\partial u^{r}}
        \right|_{u=1}
    }{
        [z^{n}]
        F_{1}(z)
    }
    ,
\end{equation}
where \( (x)_{r} \coloneqq x(x-1)\cdots(x-r+1)\).
It is not difficult to use induction and
singularity analysis of the partial derivatives
$\partial^{r} F_{1}(z,u)/\partial u^{r}\big|_{u=1}$
to show that
\begin{equation}\label{part}
    [z^{n}]
    \left.
    \frac{\partial^{r} F(z,u)}{\partial u^{r}} \right|_{u=1}
    =
    \left( 
        1 + O\left( n^{-1} \right)
    \right)
    \lambda_{1}
    \left( 
    \frac{2}{3}
    \right)^{r}
    n^{r-\frac{3}{2}}
    \rho_{1}^{-n}
    ,
\end{equation}
where \(\lambda_{1}\) is as in \eqref{eq:a:n:1}.
Thus it follows from the estimation of \(\ai_n\) in \eqref{eq:a:n:1} that
\begin{align}
    \E (X_{n})_{r}
    =
    \left( 
        1 + O\left( n^{-1} \right)
    \right)
    \left( 
    \frac{2n}{3}
    \right)^{r}
    .
    \label{eq:e:Sn}
\end{align}
Computing one more terms for \(\e{(X_{n})_{2}}\) yields, 
omitting the calculations,
\begin{equation}
    \label{varS}
    \Var{X_{n}} \sim
    %\frac{2}{9}{\frac {121635-70226\sqrt {3}}{191861\sqrt {3}-332313}}
    \frac{1+\sqrt{3}}{9}
    n
    \eqqcolon
    \gssx n
    .
\end{equation}
\end{example}

The asymptotic estimates \eqref{eq:e:Sn} and \eqref{varS}  suggest, but are
not enough to conclude, 
asymptotic normality.
However, we can apply the following general theorem.
\begin{theorem}
    \label{Tsize}
Assume \(\Phi(z)\) satisfies the conditions of \refT{Tmom2}. Let \(\rho_{1},
s_{1}\) and \(F_{0}(z)\) be as in \refS{sec:mmt:proof}. Let
\begin{equation}
    d_{1} \coloneqq F_{0}'(\rho_{1}),
    \quad
    d_{2} \coloneqq F_{0}''(\rho_{1}),
    \quad
    d_{3} \coloneqq \Phi''(s_{1} + F_{0}(\rho_{1})),
    \label{hddd}
\end{equation}
and
\begin{align}
    \mux 
    %&
    \coloneqq
    \frac{
        s_{1}
    }{s_{1}+\rho_{1}d_{1}}
    ,
    \quad
    %\label{mu:X}
    %\\
    \gssx
    %&
    \coloneqq
    \frac 
    {
        \rho_{1}  \left( d_{2} d_{3} \rho_{1} s_{1}^{2}
            +
            d_{1}^{2} d_{3} \rho_{1} s_{1}
            +
            d_{1} d_{3}  s_{1}^{2}
            -
            d_{1}^{2} \right) 
    }{ 
        \left( d_{1}\rho_{1} +s_{1} \right)^{3}d_{3}
    }
    .
    \label{mu:sigma}
\end{align}
Then \(0 < \mux < 1 \) and
\begin{equation}\label{EVX}
    \e{X_{n}}= \mux n+O(1),
    \qquad
    \Var{X_{n}} =\gssx n+O(1).
\end{equation}
Furthermore, as \ntoo,
\begin{equation}
    \frac{X_{n}-\mux n}{\sqrt{n}}
    \dto
    N(0, \gssx)
    ,
    \label{CLT:Sn:01}
    \end{equation}
and if \(\gssx \ne 0\),
then also
\begin{equation}
    \frac{X_{n}-\e{X_{n}}}{\sqrt{\Var{X_n}}}
    \dto
    N(0, 1)
    .
    \label{CLT:Sn:1}
    \end{equation}
\end{theorem}

\begin{proof}
%To show this, we apply the Quasi-Power method, in particular Theorem 2.23 of
%\cite{drmota2009random}.
Let \(\tdF(x, y, u) \coloneqq x u \Phi(y + F_{0}(x))\).
Then most conditions of  \cite[Theorem 2.23]{drmota2009random} are satsifed
by the arguments in
\refS{sec:mmt:proof}, in particular \eqref{eqgam'}.
We only need to verify that
\begin{equation}
    \label{tdF}
    \frac{
        \partial \tdF
    }{
        \partial x
    }
    \left( 
        \rho_{1},
        s_{1},
        1
    \right)
    \ne
    0
    ,
    \qquad
    \frac{
        \partial^2 \tdF
    }{
        \partial y^2
    }
    \left( 
        \rho_{1},
        s_{1},
        1
    \right)
    \ne
    0
    ,
\end{equation}
where \(\rho_{1}\) and \(s_{1}\) are as in
\refS{sec:mmt:proof}.
The first inequality holds since
\begin{equation}
    \frac{
        \partial \tdF
    }{
        \partial x
    }
    \left( 
        \rho_{1},
        s_{1},
        1
    \right)
    \ge
    \Phi\left( 
        s_{1}
        +
        F_{0}(\rho_1)
    \right)
    =
    \frac{
        s_{1}
    }{
        \rho_{1}
    }
    >
    0
    .
    \label{tdF1}
\end{equation}
By condition \eqref{philm1}, \(\Phi''(z) > 0\) for all \(z > 0\).
Then the second inequality of \eqref{tdF} holds since
\begin{equation}
    \frac{
        \partial^2 \tdF
    }{
        \partial y^2
    }
    \left( 
        \rho_{1},
        s_{1},
        1
    \right)
    =
    \rho_{1}
    \Phi''\left( s_{1} + F_{0}(\rho_1) \right)
    >
    0
    .
    \label{tdF2}
\end{equation}

Thus \cite[Theorem 2.23]{drmota2009random} applies,
which yields \eqref{EVX} and \eqref{CLT:Sn:1}, with 
\eqref{CLT:Sn:01} as an immediate consequence;
it also gives formulas for \(\mux\) and \(\gssx\), which after some
calculations yield \eqref{mu:sigma}.
\end{proof}

\begin{example}
    In the case of full binary trees, it is easy to verify with Maple that,
see also \eqref{rho1a},
    \begin{equation}
        \rho_{1}
        =
        \frac{1}{2}\sqrt {2\sqrt {3}-3}
        ,
        \qquad
        s_{1}
        =
            \frac{
                \sqrt{
                    2
                }
            }{
                \sqrt[4]{3}
            }
        ,
    \end{equation}
    and
    \begin{equation}
        d_{1}
        =
        1+\frac{\sqrt{3}}{3}
        ,
        \qquad
        d_{2}
        =
        \frac{1}{3} \sqrt{144 + 86 \sqrt{3}}
        ,
        \qquad
        d_{3}
        =2
        .
        \label{ddd2}
    \end{equation}
    Plugging these numbers into \eqref{mu:sigma}, we recover 
\eqref{eq:e:Sn} ($r=1$) and \eqref{varS}, and also
\begin{equation}
    \frac{X_{n}-2n/3}{\sqrt{n}}
    \dto
%    N(0, \gssx)
    N\Bigpar{0, \frac{1+\sqrt3}{9}}    
    .
    \label{CLT:Sn}
\end{equation}
\end{example}

%\section*{Acknowledgement}
%This work  was partially supported by 
%a grant from the Knut and Alice Wallenberg Foundation.

\newcommand\AAP{\emph{Adv. Appl. Probab.} }
\newcommand\JAP{\emph{J. Appl. Probab.} }
\newcommand\JAMS{\emph{J. \AMS} }
\newcommand\MAMS{\emph{Memoirs \AMS} }
\newcommand\PAMS{\emph{Proc. \AMS} }
\newcommand\TAMS{\emph{Trans. \AMS} }
\newcommand\AnnMS{\emph{Ann. Math. Statist.} }
\newcommand\AnnPr{\emph{Ann. Probab.} }
\newcommand\CPC{\emph{Combin. Probab. Comput.} }
\newcommand\JMAA{\emph{J. Math. Anal. Appl.} }
\newcommand\RSA{\emph{Random Struct. Alg.} }
\newcommand\ZW{\emph{Z. Wahrsch. Verw. Gebiete} }
\newcommand\DMTCS{\jour{Discr. Math. Theor. Comput. Sci.} }

\newcommand\AMS{Amer. Math. Soc.}
\newcommand\Springer{Springer-Verlag}
\newcommand\Wiley{Wiley}

\newcommand\vol{\textbf}
\newcommand\jour{\emph}
\newcommand\book{\emph}
\newcommand\inbook{\emph}
\def\no#1#2,{\unskip#2, no. #1,} %(typeset after year) 
\newcommand\toappear{\unskip, to appear}

\newcommand\arxiv[1]{\texttt{arXiv:#1}}
\newcommand\arXiv{\arxiv}

\def\nobibitem#1\par{}

\end{document}